\documentclass[11pt]{amsart}

\usepackage[
paper=a4paper,
text={147mm,230mm},centering
]{geometry}

\iftrue
\makeatletter
\def\@settitle{%
  \vspace*{-20pt}
  \begin{flushleft}%
    \baselineskip14\p@\relax
    \normalfont\bfseries\LARGE
    \@title
  \end{flushleft}%
}
\def\@setauthors{%
  \begingroup
  \def\thanks{\protect\thanks@warning}%
  \trivlist
  \large \@topsep30\p@\relax
  \advance\@topsep by -\baselineskip
  \item\relax
  \author@andify\authors
  \def\\{\protect\linebreak}%
  \authors
  \ifx\@empty\contribs
  \else
    ,\penalty-3 \space \@setcontribs
    \@closetoccontribs
  \fi
  \normalfont
  \endtrivlist
  \endgroup
}
\def\@setabstracta{%
    \ifvoid\abstractbox
  \else
    \skip@25\p@ \advance\skip@-\lastskip
    \advance\skip@-\baselineskip \vskip\skip@
    \box\abstractbox
    \prevdepth\z@ 
    \vskip-10pt
  \fi
}
\renewenvironment{abstract}{%
  \ifx\maketitle\relax
    \ClassWarning{\@classname}{Abstract should precede
      \protect\maketitle\space in AMS document classes; reported}%
  \fi
  \global\setbox\abstractbox=\vtop \bgroup
    \normalfont\small
    \list{}{\labelwidth\z@
      \leftmargin0pc \rightmargin\leftmargin
      \listparindent\normalparindent \itemindent\z@
      \parsep\z@ \@plus\p@
      
    }%
    \item[\hskip\labelsep\bfseries\abstractname.]%
}{%
  \endlist\egroup
  \ifx\@setabstract\relax \@setabstracta \fi
}
\def\section{\@startsection{section}{1}%
  \z@{-1.2\linespacing\@plus-.5\linespacing}{.8\linespacing}%
  {\normalfont\bfseries\large}}
\def\subsection{\@startsection{subsection}{2}%
  \z@{-.8\linespacing\@plus-.3\linespacing}{.3\linespacing\@plus.2\linespacing}%
  {\normalfont\bfseries}}
\def\subsubsection{\@startsection{subsubsection}{3}%
  \z@{.7\linespacing\@plus.1\linespacing}{-1.5ex}%
  {\normalfont\itshape}}
\def\@secnumfont{\bfseries}
\makeatother
\fi 

\usepackage{amssymb}
\usepackage{mathrsfs}
\usepackage[all]{xy}
\usepackage[bookmarks]{hyperref}
\usepackage{pinlabel}
\usepackage[textwidth=1.3in,color=yellow]{todonotes}
\usepackage{amsmath}
\usepackage{enumitem}
\usepackage{pb-diagram,pb-xy}
\usetikzlibrary{calc}
\usepackage{youngtab}
\usepackage[boxsize=.3 em]{ytableau}

\usepackage{mathtools}
\usepackage{url}
\usepackage{caption}
\usepackage{subcaption}
\usepackage{fancybox}
\usepackage{wrapfig}
\usepackage{color}
\usepackage{multirow, url}
\usepackage{tikz}
\usetikzlibrary{shapes}
\usepackage{xcolor}

\textwidth=\paperwidth \advance\textwidth by-2\oddsidemargin
\advance\oddsidemargin by-1in
\evensidemargin=\oddsidemargin
\calclayout

\theoremstyle{plain}
\newtheorem{theorem}{Theorem}[section]
 
\newtheorem*{thmx}{Theorem}

\newtheorem{proposition}[theorem]{Proposition}
\newtheorem{lemma}[theorem]{Lemma}

\theoremstyle{definition}
\newtheorem{question}[theorem]{Question}

\newtheorem{definition}[theorem]{Definition}
\newtheorem{example}[theorem]{Example}

\theoremstyle{remark}
\newtheorem{remark}[theorem]{Remark}

\newcommand{\ddge}{\rotatebox[origin=c]{270}{$\ge$}}

\newcommand{\C}{\mathbb{C}}

\newcommand{\R}{\mathbb{R}}

\newcommand{\Z}{\mathbb{Z}}

\newcommand{\CP}{\mathbb{C}P}

\def\pa{\partial}
\def\mcal{\mathcal}
\def\frak{\mathfrak}
\def\scr{\mathscr}

\numberwithin{equation}{section} \numberwithin{table}{section}

\def\to{\mathchoice{\longrightarrow}{\rightarrow}{\rightarrow}{\rightarrow}}
\makeatletter
\newcommand{\shortxra}[2][]{\ext@arrow 0359\rightarrowfill@{#1}{#2}}
\def\longrightarrowfill@{\arrowfill@\relbar\relbar\longrightarrow}
\newcommand{\longxra}[2][]{\ext@arrow 0359\longrightarrowfill@{#1}{#2}}

\makeatother
\numberwithin{equation}{section}

\begin{document}                                                                          
\title[Non-displaceable Lagrangian submanifolds in two-step flag varieties]{On non-displaceable Lagrangian submanifolds in two-step flag varieties}

\author{Yoosik Kim}
\address{Department of Mathematics and Institute of Mathematical Science, Pusan National University}
\email{yoosik@pusan.ac.kr}

\begin{abstract}
We prove that the two-step flag variety $\mcal{F}\ell(1,n;n+1)$ carries a non-displaceable and non-monotone Lagrangian Gelfand--Zeitlin fiber diffeomorphic to $S^3 \times T^{2n-4}$ and a continuum family of non-displaceable Lagrangian Gelfand--Zeitlin torus fibers when $n > 2$. 
\end{abstract}


\maketitle
\setcounter{tocdepth}{1} 
\tableofcontents

\section{Introduction}
\label{secIntroduction}

In his celebrated work \cite{Floerinters}, Floer invented Lagrangian Floer homology to prove a version of Arnold conjectures, every compact Lagrangian submanifold that does \emph{not} bound any non-constant holomorphic disk \emph{cannot} be displaced by any Hamiltonian diffeomorphisms. This result alludes that a symplectic manifold carries a Lagrangian submanifold having additional intersection properties and leads to the following definition.

\begin{definition}
A Lagrangian submanifold $L$ is called \emph{non-displaceable} by a Hamiltonian diffeomorphism, that is, $L \cap \phi(L) \neq \emptyset$ for any $\phi \in \mathrm{Ham}(M, \omega)$.
\end{definition}

In fact, even if a Lagrangian submanifold $L$ of $(X, \omega)$ bounds a holomorphic disk, $L$ can be non-displaceable. The fact motivates the development of Lagrangian Floer theory in a general setting and it has been an interesting question in symplectic topology to construct and detect non-displaceable Lagrangian submanifolds. The goal of this paper is to add new non-displaceable Lagrangian submanifolds to the list.

One of the main techniques to show non-displaceability is to deform Lagrangian Floer theory. Floer cohomology and its deformations by cycles and ambient cycles were developed and Floer cohomology of toric fibers on toric manifolds was computed, see \cite{ChoOh, FOOOToric1,FOOOToric2, ChoPoddar, Woodward} for instance. One of the interesting features is that a certain symplectic toric manifold/orbifold admits a continuum of non-displaceable Lagrangian toric fibers. The positions of non-displaceable toric fibers in a compact toric manifold detected by deformed Floer cohomology can be characterized as an intersection of tropicalizations in \cite{KimLeeSanda}. 

A partial flag manifold has been explored as it can be understood via the central toric variety of a toric degeneration. Also, each partial flag manifold $\mathrm{SL}(n, \C)/P$ carries a Gelfand--Zeitlin completely integrable system which shares many nice properties that a toric moment map has, see \cite{GuilleminSternbergGC}. For instance, each component generates a Hamiltonian circle action on an open dense subset and the image is a polytope and the fiber over each point in the interior of the polytope is a Lagrangian torus. The disk potential function for the Gelfand--Zeitlin system of a partial flag manifold was computed in \cite{NishinouNoharaUeda}. As a byproduct, one obtains non-displaceability of the Lagrangian torus at the ``center". Interestingly, unlike the case of toric moment maps, the Gelfand--Zeitlin system often possesses non-torus Lagrangian fibers at a stratum of the polytope. The topology of non-torus fibers was recently studied, see \cite{ChoKimOhLag, BoulocMirandaZung, CarlsonLane}. The following question can be taken into account.

\begin{question}\label{question2}
Classify all non-displaceable Lagrangian Gelfand--Zeitlin fibers. Which Lagrangian Gelfand--Zeitlin non-torus fiber is non-displaceable? 
\end{question}

When a partial flag manifold is a complex projective space, the Gelfand--Zeitlin system is a toric moment map so that the system does not have any Lagrangian non-toric fibers and there exists a unique non-displaceable fiber at the center by combining \cite{ChoSpin, McDuff}. Let us denote by $\CP(V)$ the complex projective space of a complex vector space $V$. Besides the complex projective spaces, the only known case is a complete flag variety 
$$
\mcal{F}\ell(1,2;3) = \{ (0 \subset V_1 \subset V_2 \subset \C^3) \mid \dim_\C V_j = j \mbox{ for } j = 1, 2 \} \subset \CP(\C^3) \times \CP( \wedge^2 \C^3)
$$ 
by combining \cite{ChoKimOhLGbulk, Pabiniak}. More precisely, if $\mcal{F}\ell(1,2;3)$ is equipped with a \emph{non-monotone} pull-backed symplectic form from a product of Fubini--Study forms, then the Lagrangian torus fiber at the center is the \emph{only} non-displaceable fiber. Suppose that $\mcal{F}\ell(1,2;3)$ is equipped with a \emph{monotone} pull-backed symplectic form from a product of Fubini--Study forms. In this case, the Gelfand--Zeitlin system has a \emph{unique} non-torus Lagrangian fiber and this fiber is diffeomorphic to a $3$-sphere. Let $I$ be the line segment joining the location of the $S^3$-fiber and the center of the image of the Gelfand--Zeitlin system. The fiber over every point of $I$ is non-displaceable and the fiber over every point in the complement of $I$ is displaceable.
 
Some partial answers are given in some cases. A non-displaceable Lagrangian $\mathrm{U}(n)$-fiber in $\mathrm{Gr}(n, {2n})$ was shown to be non-displaceable in \cite{NoharaUedanontorus, EvansLekili}. For a general complete flag manifold $\mcal{F}\ell(1, 2, \cdots, n ;n+1)$ with $n \geq 3$, the polytope $\Delta$ has multiple line segments joining the center of $\Delta$ to the center of a certain face at which a monotone Lagrangian non-toric Lagrangian submanifold is located. The fiber at each point of the union of line segments was shown to be non-displaceable in \cite{ChoKimOhLGbulk}. In particular, $\mcal{F}\ell(1, 2, \cdots, n ;n+1)$ admits non-displaceable Lagrangian submanifolds diffeomorphic to a product of unitary group and torus.

One of the reasons why complete flag varieties were explored for the first time in \cite{ChoKimOhLGbulk} is something to do with the dimension of the space of deformations for Lagrangian Floer theory. The dimension of the space of deformations by ambient cycles of codimension two is the second betti number. As we have the larger second betti number, we have more room to deform Lagrangian Floer theory. In this regard, Grassmannians are the most difficult spaces as the second betti number is one. 

In this article, we tackle the next case where the second betti number is two. We explore non-displaceable Lagrangians in the two-step flag varieties 
$$
\mcal{F}\ell(1,n;n+1) = \{ (0 \subset V_1 \subset V_{n} \subset \C^{n+1}) \mid \dim_\C V_j = j \mbox{ for } j = 1, n \} \to \CP(\C^{n+1}) \times \CP\left(\wedge^n \C^{n+1} \right)
$$ 
equipped a pull-backed symplectic form from a \emph{monotone} product of Fubini--Study form. Here, the monotonicity means that $\lambda = \lambda^\prime$ in the product symplectic form $\lambda \cdot \omega_{\CP(\C^{n+1})} \oplus \lambda^\prime \cdot \omega_{\CP\left(\wedge^n \C^{n+1} \right)}$. Let the unitary group $G \coloneqq \mathrm{U}(n+1)$ act on the dual Lie algebra $\frak{g}^*$ of its Lie algebra $\frak{g}$ via the coadjoint representation. Let $E_{i,j}$ be the matrix whose $(i,j)$ entry is one and the other entries are all zero. Let $\epsilon_j$ be the linear map  on the set of diagonal matrices defined by the assignment $\epsilon_j (E_{ii}) = \delta_{ij}$. Let $\varpi_i \coloneqq \epsilon_1 + \cdots + \epsilon_i$ be the $i$-th fundamental weight. 

\begin{thmx}[Theorem~\ref{theorem_main} and~\ref{theorem_GZtorusfibernon}]
For $n \geq 3$, choose $\lambda = \varpi_1 + \varpi_n \in \frak{g}^*$ and consider the coadjoint orbit $\mcal{O}_\lambda$ of $\lambda$ equipped with a Kirillov--Kostant--Souriau symplectic form, which is isomorphic to the two-step flag variety $\mathcal{F}\ell(1, n; n+1)$ with the pull-backed form. Let $\Phi_\lambda \colon \mcal{O}_\lambda \to \R^{2n-1}$ be the Gelfand--Zeitlin system. Then it carries a continuum of non-displaceable Lagrangian tori and a non-displaceable non-monotone non-toric Lagrangian submanifold diffeomorphic to $S^3 \times T^{2n-4}$.
\end{thmx}

\begin{remark}
The non-displaceability of non-toric Lagrangian submanifold diffeomorphic to $S^3 \times T^{2n-4}$ follows from the fact that this non-toric Lagrangian is realized as the limit of non-displaceable Lagrangians, see  Lemma~\ref{lemma_limitnondisplace}. 
\end{remark}

Previously, in the complete flag manifolds \cite{ChoKimOhLGbulk}, non-displaceable Lagrangian tori are located at a line segment joining the center of the polytope to the center of a face. In other words, the one-parameter family of non-displaceable fibers is obtained by connecting two monotone Lagrangian fibers. However, in the two-step flag variety $\mcal{F}\ell(1,n;n+1)$ equipped with a monotone symplectic form, non-displaceable Lagrangian submanifolds occur as a family of connecting monotone Lagrangian torus and \emph{non-monotone} non-torus Lagrangian fiber. The result implies that Lagrangians located at the line segment joining two centers are not the only candidates for positions of non-displaceable Lagrangians. Question~\ref{question2} is largely open and it seems that new ideas and systematic methods are necessary to tackle the question.

\subsection*{Acknowledgments}

The author is grateful to Yunhyung Cho and Yong-Geun Oh for their valuable discussion and collaboration. This paper starts from the collaboration \cite{ChoKimOhLag, ChoKimOhLGbulk}. This work was supported by a 2-Year Research Grant of Pusan National University.

\section{Preliminaries}\label{secReview}

In this section, we review some facts on two-step flag varieties $\mathcal{F}\ell(1,n;n+1)$ and Gelfand--Zeitlin systems.

\subsection{Two-step flag varieties $\mathcal{F}\ell(1,n;n+1)$}

For a natural number $n \geq 2$, the two-step flag variety $\mathcal{F}\ell(1,n;n+1)$ is defined by 
\begin{equation}\label{equ_mathcalfell1nn+1}
\mathcal{F}\ell(1,n;n+1) = \left\{ (0 \subset V_1 \subset V_2 \subset \C^{n+1}) \mid \dim_\C V_1 = 1 \mbox{ and } \dim_\C V_2 = n \right\}.
\end{equation}
The complex Lie group $G_\C \coloneqq \mathrm{SL}(n+1,\C)$ acts on $\C^{n+1}$ linearly. The linear action induces a transitive action on $\mcal{F}\ell(1,n;n+1)$ and the isotropy subgroup of the standard flag $\langle e_1 \rangle \subset \langle e_1, \cdots, e_n \rangle$ is a parabolic subgroup $P$ of $G_\C$. Thus, $\mcal{F}\ell(1,n;n+1)$ is a homogeneous space $G_\C / P$ of complex dimension,
$$
\dim_\C \mathcal{F}\ell(1,n;n+1)  = \dim_\C G_\C - \dim_\C P =  2n - 1.
$$

The two-step flag variety can be embedded into a product of projective spaces via the Pl\"{u}cker embedding
\begin{equation}\label{equ_pluckerembedding}
\mathcal{F}\ell(1,n;n+1) \to \mathbb{P} \left(\C^{n+1} \right) \times \mathbb{P} \left(\wedge^n \C^{n+1} \right) \simeq \CP^{n} \times \CP^{n}, \quad (V_1, V_2) \mapsto \left(V_1, \wedge^n V_2 \right).
\end{equation}
Let us equip $\mathcal{F}\ell(1,n;n+1)$ with the pull-back K\"{a}hler form via~\eqref{equ_pluckerembedding} induced from a multiple 
\begin{equation}\label{equ_multipleprodfsfs}
(\lambda_1 - \lambda_2) \cdot \omega_{\mathrm{FS}} \oplus \left( \lambda_2 - \lambda_{3} \right) \cdot \omega_{\mathrm{FS}} 
\end{equation}
of the product of Fubini--Study forms on $\CP^{n} \times \CP^{n}$ for some $\lambda_1, \lambda_{2}, \lambda_3 \in \mathbb{Z}$ with \begin{equation}\label{equ_monotonechoicegen}
\lambda = (\lambda_1 >  \lambda_2  > \lambda_{3}).
\end{equation}

\begin{remark}
For the expression~\eqref{equ_multipleprodfsfs}, the differences of $\lambda_\bullet$'s only matter. If $\lambda_1 - \lambda_2 = \lambda_1^\prime - \lambda_2^\prime$ and $\lambda_2 - \lambda_3 = \lambda_2^\prime - \lambda_3^\prime$,  the symplectic forms are same. A reason why the symplectic form is expressed in that way is to compare with the choice~\eqref{equ_lambdachoice} for the coadjoint orbit $\mcal{O}_\lambda$ and its associated invariant symplectic form. 
\end{remark}

In order to represent a point in $\mathcal{F}\ell(1,n;n+1)$, consider an invertible $(n+1) \times (n+1)$ matrix $Z = [z_{i,j}]_{1 \leq i, j \leq n+1}$ with complex entries. The flag corresponding to the matrix $Z$ consists of subspaces generated by the first column and the first $n$ columns of $Z$ in $\C^{n+1}$. For $I = (i_1, i_2, \cdots, i_k)$ with $I \subset \{1, 2, \cdots, n+1 \}$, the Pl\"{u}cker variable $p_I$ is defined by the minor
$$
{p}_{I} \coloneqq \det 
\begin{bmatrix}
z_{i_1,1} & z_{i_1,2} & \cdots & z_{i_1,k} \\
z_{i_2,1} & z_{i_2,2} & \cdots & z_{i_2,k} \\
\vdots & \vdots & \vdots& \vdots\\ 
z_{i_k,1} & z_{i_k,2} & \cdots & z_{i_k,k} 
\end{bmatrix}.
$$
Setting
$$
i = (i) \mbox{ and } \underline{i} \coloneqq ( 1, 2, \cdots, i-1, i+1, \cdots, n+1 ),
$$
the condition $V_1 \subset V_2$ yields the Pl\"{u}cker relation
\begin{equation}\label{equ_pluckerrel}
p_1 {p}_{\underline{1}} - p_2 {p}_{\underline{2}} + p_3 {p}_{\underline{3}} + \cdots + (-1)^{n} p_{n+1} {p}_{\underline{n+1}} = 0.
\end{equation}
In fact, the embedded variety is the bidegree $(1,1)$ hypersurface defined by~\eqref{equ_pluckerrel} in $\CP^{n} \times \CP^{n}$. 

The Pl\"{u}cker variables $p_I$'s constitute a SAGBI basis for an (anti-)diagonal term order by \cite[Theorem 5]{KoganMiller}. It in turn implies that there is a flat morphism $\Pi \colon \mcal{X} \to \C$ induced by the projection $ \CP^n \times \CP^n \times \C \to \C$ where
\begin{equation}\label{equ_toricdegenerations}
\mcal{X} \coloneqq V \left( p_1 {p}_{\underline{1}} - p_2 {p}_{\underline{2}} + t \left( p_3 {p}_{\underline{3}} + \cdots + (-1)^{n} p_{n+1} {p}_{\underline{n+1}} \right) \right) \subset \CP^n \times \CP^n \times \C.
\end{equation}
Letting $\mcal{X}_t \coloneqq \Pi^{-1}(t)$, the fiber $\mcal{X}_1$ is the exactly the embedded variety given by~\eqref{equ_pluckerrel}. Also, the central fiber $\mcal{X}_0$ is a toric variety as it is defined by a homogeneous toric (binomial) ideal. The toric variety $\mcal{X}_0$ is singular when $n \geq 2$ and its singular locus $\mathrm{Sing}(\mcal{X}_0)$ is exactly the subvariety of complex codimension three given by $p_1 = p_2 = {p}_{\underline{1}} = {p}_{\underline{2}} = 0$. 

To describe the torus action on the toric variety $\mcal{X}_0$, consider the real torus with coordinates
$$
\theta_{1{2}}, \theta_{1\underline{2}}, \theta_{\underline{1}{2}}, \theta_{\underline{1}\underline{2}}, \theta_3, \theta_{\underline{3}}, \cdots, \theta_n, \theta_{\underline{n}}
$$
which acts on $\CP^n \times \CP^n$ as
\begin{equation}\label{equ_hamiltoniantnaction}
\begin{cases}
\displaystyle \theta_{1{2}} \mapsto \left( e^{\sqrt{-1} \theta_{1{2}}} p_{\vphantom{\underline{1}}1} , e^{\sqrt{-1} \theta_{1{2}}} p_{\vphantom{\underline{2}}2} \right), \,
\theta_{1\underline{2}} \mapsto \left( e^{\sqrt{-1} \theta_{1\underline{2}}} p_{\vphantom{\underline{1}}1} , e^{\sqrt{-1} \theta_{1\underline{2}}} p_{\underline{2}} \right),  \,
\theta_{\vphantom{\underline{j}}j} \mapsto e^{\sqrt{-1} \theta_{j}} p_{\vphantom{\underline{j}}{j}}, \\
\displaystyle \theta_{\underline{1}2} \mapsto \left( e^{\sqrt{-1} \theta_{\underline{1}2}} p_{\underline{1}}, e^{\sqrt{-1} \theta_{\underline{1}2}} p_{\vphantom{\underline{1}}2} \right), \,
\theta_{\underline{1}\underline{2}} \mapsto \left( e^{\sqrt{-1} \theta_{\underline{1}\underline{2}}} p_{{\underline{1}}}, e^{\sqrt{-1} \theta_{\underline{1}\underline{2}}} p_{\underline{2}} \right),  \,
\theta_{\underline{j}} \mapsto e^{\sqrt{-1} \theta_{\underline{j}}} p_{\underline{j}}.
\end{cases}
\end{equation}
Note that the torus action induces that on the toric variety $\mcal{X}_0$. By taking a linear combination of the torus action~\eqref{equ_hamiltoniantnaction}, we may take a moment map (with respect to a suitable multiple of the K\"{a}hler form~\eqref{equ_multipleprodfsfs}) for the toric $T^{2n-1}$-action on $\mcal{X}_0$ as follows$\colon$
\begin{equation}\label{equ_coorGC}
\Phi_0 \colon \mcal{X}_0 \to \R^{2n-1} = \R \langle u_{1,n}, u_{1,n-1}, \cdots, u_{1,2}, u_{1,1}, u_{2,1}, \cdots, u_{n-1,1}, u_{n,1} \rangle
\end{equation}
defined by 
\begin{equation}\label{equ_momentmap}
\begin{cases}
\displaystyle u_{1,j} =  \lambda_2 +  (\lambda_1 - \lambda_2) \cdot \sum_{i=1}^j  \frac{| p_{\vphantom{\underline{1}}i} |^2}{\| \mathbf{\mathbf{p}} \|^2} , \quad \mbox{for $j= 2, 3, \cdots, n$},  \\
\displaystyle u_{1,1} =  \lambda_2 + (\lambda_1 - \lambda_2) \cdot \frac{| p_{\vphantom{\underline{1}}1} |^2}{\| \textbf{\textup{p}} \|^2} + (\lambda_3 - \lambda_{2}) \cdot \frac{| p_{{\underline{1}}} |^2}{\| \underline{\mathbf{\mathbf{p}}} \|^2} ,  \, \\
\displaystyle u_{j,1} = \lambda_2 + (\lambda_3 - \lambda_{2}) \cdot \sum_{i=1}^j  \frac{| p_{{\underline{i}}} |^2}{\| \underline{\mathbf{\mathbf{p}}} \|^2},   \quad \mbox{for $j= 2, 3, \cdots, n$}, \, 
\end{cases}
\end{equation}
where $\| \mathbf{p} \|^2 = \sum_{i=1}^{n+1} |p_i|^2$ and $\| \underline{\mathbf{p}} \|^2 = \sum_{i=1}^{n+1} |{p}_{\underline{i}}|^2$. Then the image of the toric moment map $\Phi_0$ in~\eqref{equ_coorGC} is the polytope determined by
\begin{equation}\label{equation_GZ-pattern}
\begin{alignedat}{17}
  &\quad \,\, \lambda_1 &&&&&&&&&&&&&&&&  \\
  &\quad \,\,\, \ddge &&&&&&&&&&&&&&&& \\
  &\quad u_{1, n} &&&&&&&&&&&&&&&& \\
  &\quad \,\,\, \ddge &&&&&&&&&&&&&&&& \\
  &\quad \,\,\, \vdots &&&&&&&&&&&&&&&&&&&& &&&&  \\
  &\quad \,\,\, \ddge &&&&&&&&&&&&&&&& \\
  &\quad u_{1,2} &&&& \geq &&&& \quad \,\,  \lambda_2 &&&&&&&&&&&&&&&&&&&&&&&&  \\
  &\quad \,\,\, \ddge &&&&&&&&\quad \,\,\, \ddge &&&&&&&& &&&&&&&&&&&&&&&&  \\
  &\quad u_{1,1} &&&& \geq &&&& \quad  u_{2,1} \quad &&&& \geq &&&& \quad  \cdots \quad &&&& \geq &&&&&&&&&&&& \quad u_{n-1,1} \quad &&&& \geq &&&& \quad \, \lambda_{3}.
\end{alignedat}
\end{equation}

This polytope given by~\eqref{equation_GZ-pattern} in $M_\R \simeq \R^{2n-1}$ is called the \emph{Gelfand--Zeitlin} polytope and is denoted by $\Delta_\lambda$. The singular locus $\mathrm{Sing}(\mcal{X}_0)$ is located at the stratum contained in $u_{1,1} = u_{1,2} = u_{2,1} = \lambda_2$ by examining the image of $\mathrm{Sing}(\mcal{X}_0)$ via the map~\eqref{equ_momentmap}.

\subsection{The Gelfand--Zeitlin polytope of $\mathcal{F}\ell(1,n;n+1)$}

We describe the face structure of the Gelfand--Zeitlin (GZ for short) polytope in terms of combinatorics of an associated ladder diagram. Consider the grid $(\Z \times \R) \cup (\R \times \Z)$ in $\R^2$ and let $\square^{(i,j)}$ be the unit box whose vertices are located at $(i,j), (i-1,j), (i,j-1),(i-1,j-1)$. In order to study $\mcal{F}\ell(1,n;n+1)$ and $\Delta_\lambda$, we take the ladder diagram associated to $\lambda$ as
$$
\Gamma_n \coloneqq \bigcup_{i=1}^{n} \left( \square^{(1,i)} \cup \square^{(i,1)} \right).
$$
There are two farthest points from the origin located at $(1, n)$ and $(n, 1)$ in $\Gamma_n$. A positive path is a shortest path from the origin to either $(1,n)$ or $(n,1)$ contained in $\Gamma_n$. Thus, the length of every positive path is $n + 1$. Let $\scr{P}(\Gamma_n)$ be the set of subgraphs that can be expressed as a union of positive paths of $\Gamma_n$ and contain both farthest points $(1, n)$ and $(n, 1)$. It is endowed with a partial ordering given by the set-theoretical inclusion $\subset$. 

The partially ordered set $\left(\scr{P}(\Gamma_n), \subset \right)$ can be used to describe the face structure of $\Delta_\lambda$. Let us overlap the ladder diagram $\Gamma_n$ and the pattern~\eqref{equation_GZ-pattern}. We fill each unit box $\square^{(i,j)}$ with the component $u_{i,j}$ and locate $\lambda_1, \lambda_2$, and $\lambda_3$ at $\square^{(1,n+2)}, \square^{(2,2)}$, and $\square^{(n+2,1)}$, respectively.  For each subgraph $\Gamma \in \scr{P}(\Gamma_n)$, the corresponding face $f_\Gamma$ is contained in the intersection of all hyperplanes given by equating two adjacent fillings that are \emph{not} divided by any segments of $\Gamma_n$, see Example~\ref{example_fillingladderd}.

\begin{theorem}[\cite{AnChoKim}]\label{theorem_AnChoKim}
Let $\scr{F}(\Delta_\lambda)$ be the set of faces of $\Delta_\lambda$ with $\lambda$ is in~\eqref{equ_monotonechoicegen} and $\scr{P}(\Gamma_n)$  the set of subgraphs which can be expressed as a union of positive paths of $\Gamma_n$.
Then there is a one-to-one correspondence 
$$
\Psi \colon \scr{P}(\Gamma_n) \to \scr{F}(\Delta_\lambda) \quad \mbox{ defined by} \quad  \Gamma \mapsto f_\Gamma
$$
such that
\begin{itemize}
\item \textup{(Order preserving)} $\Gamma \subset \Gamma^\prime \Leftrightarrow f_{\Gamma} \subset f_{\Gamma^\prime}$,
\item \textup{(Dimension)} $\mathrm{rank} \, H_1( \Gamma) = \dim f_\Gamma$.
\end{itemize}
\end{theorem}

\begin{remark}
Note that the GZ polytope $\Delta_\lambda$ depends on $\lambda$, while $\Gamma_n$ only depends on $n$. It reflects the fact that the combinatorial type of $\Delta_\lambda$ depends only on $n$ (not on $\lambda$).
\end{remark}

\begin{example}\label{example_fillingladderd}
When $n = 3$, the ladder diagram $\Gamma_3$ is depicted in the first figure of Figure~\ref{adderdia134}. The diagram will be employed to understand $\mcal{F}\ell(1,3;4)$. There are three subgraphs in Figure~\ref{adderdia134}
\begin{itemize}
\item $\Gamma$ corresponds to $f_{\Gamma}$ determined by $u_{1,2} = u_{1,1} = u_{2,1} = \lambda_2$,
\item $\Gamma^\prime$ corresponds to $f_{\Gamma^\prime}$ determined by $u_{1,3} = u_{1,2} = u_{1,1} = u_{2,1} = u_{3,1} = \lambda_2$,
\item $\Gamma^{\prime \prime}$ corresponds to $f_{\Gamma^{\prime \prime}}$ determined by $u_{1,3} = \lambda_1$.
\end{itemize}

\begin{figure}[ht]
	\scalebox{1}{
\begingroup%
  \makeatletter%
  \providecommand\color[2][]{%
    \errmessage{(Inkscape) Color is used for the text in Inkscape, but the package 'color.sty' is not loaded}%
    \renewcommand\color[2][]{}%
  }%
  \providecommand\transparent[1]{%
    \errmessage{(Inkscape) Transparency is used (non-zero) for the text in Inkscape, but the package 'transparent.sty' is not loaded}%
    \renewcommand\transparent[1]{}%
  }%
  \providecommand\rotatebox[2]{#2}%
  \newcommand*\fsize{\dimexpr\f@size pt\relax}%
  \newcommand*\lineheight[1]{\fontsize{\fsize}{#1\fsize}\selectfont}%
  \ifx\svgwidth\undefined%
    \setlength{\unitlength}{338.31783499bp}%
    \ifx\svgscale\undefined%
      \relax%
    \else%
      \setlength{\unitlength}{\unitlength * \real{\svgscale}}%
    \fi%
  \else%
    \setlength{\unitlength}{\svgwidth}%
  \fi%
  \global\let\svgwidth\undefined%
  \global\let\svgscale\undefined%
  \makeatother%
  \begin{picture}(1,0.16437966)%
    \lineheight{1}%
    \setlength\tabcolsep{0pt}%
    \put(0,0){\includegraphics[width=\unitlength,page=1]{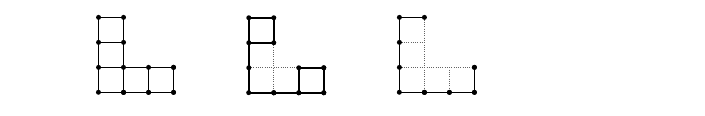}}%
    \put(0.18279045,0.01316307){\color[rgb]{0,0,0}\makebox(0,0)[lt]{\begin{minipage}{0.20099443\unitlength}\raggedright $\Gamma_3$\end{minipage}}}%
    \put(0.39574995,0.01437153){\color[rgb]{0,0,0}\makebox(0,0)[lt]{\begin{minipage}{0.25858924\unitlength}\raggedright $\Gamma$\end{minipage}}}%
    \put(0.6078076,0.0138226){\color[rgb]{0,0,0}\makebox(0,0)[lt]{\begin{minipage}{0.25858924\unitlength}\raggedright $\Gamma^\prime$\end{minipage}}}%
    \put(0,0){\includegraphics[width=\unitlength,page=2]{fig_ladderdia.pdf}}%
    \put(0.83011075,0.01316307){\color[rgb]{0,0,0}\makebox(0,0)[lt]{\begin{minipage}{0.20099443\unitlength}\raggedright $\Gamma^{\prime\prime}$\end{minipage}}}%
    \put(-0.00303085,0.09700819){\color[rgb]{0,0,0}\makebox(0,0)[lt]{\lineheight{1.25}\smash{\begin{tabular}[t]{l}\,\end{tabular}}}}%
  \end{picture}%
\endgroup%
}
	\caption{\label{adderdia134} Ladder diagram $\Gamma_3$ and three subgraphs $\Gamma$, $\Gamma^\prime$, $\Gamma^{\prime\prime}$}	
\end{figure}
\end{example}

\subsection{The Gelfand--Zeitlin system on $\mathcal{F}\ell(1,n;n+1)$}

We review a construction of the Gelfand--Zeitlin (GZ for short) system on $\mcal{F}\ell(1, n; n+1)$. To construct the GZ system on $\mcal{F}\ell(1, n; n+1)$, we first realize $\mcal{F}\ell(1,n;n+1)$ as a coadjoint orbit. Let $G \coloneqq \mathrm{U}(n+1)$ be the unitary group and let $G$ act on the linear dual $\frak{g}^*$ of its Lie algebra $\frak{g}$ via the coadjoint representation. Let 
\begin{itemize}
\item $\frak{t}$ be the Cartan subalgebra of $\frak{g}$ consisting of diagonal matrices in $G$. 
\item $E_{i,j}$ be the matrix whose $(i,j)$ entry is one and the other entries are all zero,
\item $\epsilon_j$ be the linear map  on $\frak{t}$ by the assignment $\epsilon_j (E_{ii}) = \delta_{ij}$, 
\item $\frak{t}^*$ be the linear dual of the Cartan subalgebra $\frak{t}$,
\item $\varpi_i \coloneqq \epsilon_1 + \cdots + \epsilon_i$ be the $i$-th fundamental weight.
\end{itemize}
The fundamental Weyl chamber is then given by
$$
\frak{t}^*_+ \coloneqq \bigoplus_{i=1}^{n+1} \R_{\geq 0} \varpi_i.
$$

Reflecting our choice of $\lambda$ in~\eqref{equ_monotonechoicegen}, we choose an integral dominant weight 
\begin{equation}\label{equ_lambdachoice}
\lambda = (\lambda_1 - \lambda_2) \cdot \varpi_1 + (\lambda_2 - \lambda_{3}) \cdot \varpi_n  \in \frak{t}^* \subset \frak{g}^*
\end{equation}
and the coadjoint orbit is defined by the orbit of $\lambda$ under the coadjoint $G$-action. We then have
\begin{equation}\label{equ_iso}
\mcal{O}_\lambda \simeq \frac{\mathrm{U}(n+1)}{\mathrm{U}(1) \times \mathrm{U}(n)} \simeq \frac{\mathrm{SL}(n+1,\C)}{P} \simeq \mcal{F}\ell(1,n;n+1). 
\end{equation}
The coadjoint orbit $\mcal{O}_\lambda$ naturally comes with an invariant symplectic form $\omega_\lambda$, which is called the Kirillov--Kostant--Souriau form. In fact, two forms agree via the isomorphism~\eqref{equ_iso} by our choices~\eqref{equ_multipleprodfsfs} and~\eqref{equ_lambdachoice}, see \cite[Section 2. Eq (2)]{NishinouNoharaUeda} for instance.

Now, we build a completely integrable system on $\mcal{O}_\lambda$ by applying Thimm's method.
Let $G_j \coloneqq \mathrm{U}(j)$ and $\iota_j \colon G_{j} \to G_{j+1}$ be the inclusion of $G_j$ into the leading principal submatrix, that is, $\iota_j (A) \mapsto \mathrm{diag}(A, 1)$. Let
\begin{itemize}
\item $\frak{g}^*_j$ be the dual of the Lie algebra of $G_j$,
\item $\frak{t}^*_j$ be the dual of the Cartan subalgebra of $\frak{g}_j$,
\item $\frak{t}^*_{j,+}$ be its fundamental Weyl chamber.
\end{itemize}
The coadjoint $G_j$-orbit of any $\lambda \in \frak{g}_j^*$ and the fundamental Weyl chamber $\frak{t}^*_{j,+}$  intersect at a \emph{single} point so that we have a well-defined map
$$
\Pi_j \colon \frak{g}^*_{j} \to \frak{t}^*_{j,+} \quad \lambda \mapsto (G_j \cdot \lambda) \cap  \frak{t}^*_{j,+}.
$$
We then have the following diagram
\begin{equation}\label{equ_diagram}
\xymatrix{
\mcal{O}_\lambda  \ar@{^{(}->}[r]^{\iota}  & \frak{g}^* \ar@(r,u)_{G} \ar@{->}[rr]^{\iota^*_n} &  & \frak{g}^*_{n}  \ar@(r,u)_{G_{n}} \ar[d]^{\Pi_{n}}  \ar@{->}[rr]^{\iota^*_{n-1}} && \frak{g}^*_{n-1} \ar[d]^{\Pi_{n-1}} \ar@(r,u)_{G_{n-1}}  \ar@{->}[rr] & & \cdots  \ar@{->}[rr]^{\iota^*_{1}} & & \frak{g}^*_1 \ar[d]^{\Pi_{1}} \ar@(r,u)_{G_{1}} \\
& && \frak{t}^*_{n} && \frak{t}^*_{n-1} && \cdots && \frak{t}^*_{1}.
}
\end{equation}
The first map $\iota \colon \mcal{O}_\lambda \to \frak{g}^*$ is the inclusion, which is a moment map for the coadjoint $G$-action on $\mcal{O}_\lambda$. By the functoriality of moment maps, the action of the subgroup $G_j$ of $G$ via the composition $\iota_n \circ \cdots \circ \iota_j$ has a moment map 
$$
\mcal{O}_\lambda \to \frak{g}_{j}^* \quad \mbox{given by $\iota_j^* \circ \cdots \circ \iota_n^* \circ \iota$}.
$$  
For each $j$ with $1 \leq j \leq n$, consider a map $\tau_j \colon \frak{t}^*_j \to \R^{j}$ defined by $\epsilon \mapsto \left( \epsilon(E_{ii}) \mid 1 \leq i \leq j \right)$. Set
$$
\Phi_j \colon \mcal{O}_\lambda \to \frak{g}^*_{j} \to \frak{t}^*_j \to \R^{j}, \quad \Phi_j = \tau_j \circ \Pi_j \circ \iota_j^* \circ \cdots \circ \iota_n^* \circ \iota.
$$
By collecting all $\Phi_j$, we then obtain 
$$
(\Phi_1, \cdots, \Phi_{n}) \colon \mcal{O}_\lambda \to \R^{n(n+1)/2}.
$$
By our choice of ~\eqref{equ_lambdachoice} and the interlacing inequalities arising from the min-max principle, each component $\frak{t}^*_j \to \R$ given by $\epsilon \mapsto \epsilon(E_{ii})$ for $1 < i < j$ of $\Phi_j$ restricting to $\mcal{O}_\lambda$ is a constant function. Set $\tau_j \coloneqq (\tau_{j,1}, \cdots, \tau_{j,j})$ where
$$
\tau_{j,i} \colon \frak{t}^*_j \to \R \quad \mbox{is given by $\epsilon \mapsto \epsilon (E_{ii})$}. 
$$
The non-constant components of $\Phi_j$ on $\mcal{O}_\lambda$ are denoted by
$$
\begin{cases}
\Phi_{1,j} \coloneqq \tau_{j,1} \circ \Pi_j \circ \iota_j^* \circ \cdots \circ \iota_n^* \circ \iota, \\
\Phi_{j, 1} \coloneqq \tau_{j,j} \circ \Pi_j \circ \iota_j^* \circ \cdots \circ \iota_n^* \circ \iota.
\end{cases}
$$

\begin{definition}
The \emph{Gelfand--Zeitlin system} is defined by
\begin{equation}\label{equ_GZsystemcomp}
\Phi_\lambda \coloneqq (\Phi_{1,1}, (\Phi_{1,2}, \Phi_{2,1}), \cdots, (\Phi_{1,n}, \Phi_{n,1})) \colon \mcal{O}_\lambda \to \R^{2n-1}.
\end{equation}
\end{definition}

\begin{remark}\label{remark_trans}
By interlacing inequalities, the image of the GZ system $\Phi_\lambda$ is a polytope, which is equal to the GZ polytope $\Delta_\lambda$ defined in~\eqref{equation_GZ-pattern} up to a translation of each component by $\lambda_2$. From now on, the image of the GZ system $\Phi_\lambda$ is assumed to be equal to $\Delta_\lambda$.
\end{remark}

Note that the coadjoint orbit $\mcal{O}_\lambda$ is isomorphic to $\mcal{X}_1$ as in~\eqref{equ_toricdegenerations}. Nishinou--Nohara--Ueda constructed a toric degeneration of completely integrable systems such that the GZ system in~\eqref{equ_GZsystemcomp} converges to the toric completely integrable system~\eqref{equ_coorGC}. Restricting to the case where $\mcal{F}\ell(1,n;n+1)$, their theorem is stated below.

\begin{theorem}[\cite{NishinouNoharaUeda}]\label{theorem_toricdegenerations}
Consider the Gelfand--Zeitlin system $\Phi_\lambda \colon \mcal{O}_\lambda \simeq \mcal{X}_1 \to \R^{2n-1}$ in~\eqref{equ_GZsystemcomp} and the toric moment map $\Phi_0 \colon \mcal{X}_0 \to \R^{2n-1}$ in~\eqref{equ_coorGC}. Let $\Delta_\lambda$ be the Gelfand--Zeitlin polytope in~\eqref{equation_GZ-pattern}. Then, for each $t \in [0,1]$, we have 
\begin{itemize}
\item a completely integrable system $\Phi_t \colon \mcal{X}_t \to \R^{2n-1}$ and
\item a continuous map $\psi_{1,t} \colon \mcal{X}_1 \to \mcal{X}_t$
\end{itemize}
such that
\begin{enumerate}
\item the image of $\Phi_t \colon \mcal{X}_t \to \R^{2n-1}$ is $\Delta_\lambda$ for each $t \in [0,1]$,
\item $\Phi_1$ agrees with $\Phi_\lambda$ via the isomorphism $\mcal{O}_\lambda \simeq \mcal{X}_1$,
\item $\Phi_0 \colon \mcal{X}_0 \to \R^{2n-1}$ is the toric moment map,
\item $\psi_{1,t}$ is a symplectomorphism for each nonzero $t (\neq 0)$,
\item $\psi_{1,0}$ is a symplectomorphism on the inverse image of $\mcal{X}_0 \backslash \mathrm{Sing}(\mcal{X}_0)$ under $\psi_{1,0}$. 
\item $\psi_{1,t}$ makes the following diagram commute$\colon$
\begin{equation}\label{equ_commutingdiagramgrad}
\xymatrix{
{\mcal{X}_1 \simeq \mcal{O}_\lambda} \ar[rr]^{\psi_{1,t}} \ar[rd]_{\Phi_1} & & \mcal{X}_t \ar[ld]^{\Phi_t}  \\
& \Delta_\lambda. &  }
\end{equation}
\end{enumerate}
\end{theorem}

\subsection{Topology of Gelfand--Zeitlin fibers in $\mathcal{F}\ell(1,n;n+1)$}

We now discuss the topology of fibers of the Gelfand--Zeitlin system $\Phi_\lambda$ constructed in~\eqref{equ_GZsystemcomp}. The current author with Cho and Oh in \cite{ChoKimOhLag} described each GZ fiber in terms of the total space of a certain iterated bundle. As a byproduct, every GZ fiber is shown to be a smooth isotropic submanifold. We recall their result when $\mcal{O}_\lambda \simeq \mathcal{F}\ell(1,n,n+1)$.

In this case where $\mcal{O}_\lambda \simeq \mathcal{F}\ell(1,n;n+1)$, if a point $\mathbf{u}$ in the relative interior $\mathring{f}$ of a face $f$, then the fiber is of the product form $\Phi_{\lambda}^{-1}(\mathbf{u}) \simeq \mcal{S}_f \times (S^1)^{\dim f}$ where $\mcal{S}_f$ is either a point or an odd dimensional sphere. The factor $\mcal{S}_f$ is precisely determined as follows. Set $u_{1, n+1} \coloneqq \lambda_1$ and $u_{n+1, 1} \coloneqq \lambda_{n+1}$. For an integer $j$ with $1 \leq j \leq n-1$, the condition $(j)$ is defined by
\begin{equation}\label{equ_conditionsphericalfact}
\mbox{(Condition $(j)$)} \coloneqq
\begin{cases}
u_{1,j+1}  = \cdots = u_{1,2} = u_{1,1} = u_{2,1} = \cdots =  u_{j+1,1} = \lambda_2, \\
u_{1, j+2} > u_{1, j+1}$ and $u_{j+1, 1} > u_{j+2, 1}
\end{cases}
\end{equation}
We define
\begin{equation}\label{equ_topologyofsf}
\mcal{S}_f \coloneqq
\begin{cases}
S^{2j+1} &\mbox{ if a point $\mathbf{u} \in \mathring{f}$ satisfies the condition $(j)$,} \\
\mathrm{point} &\mbox{ otherwise.} 
\end{cases}
\end{equation}

A face $f$ of $\Delta_\lambda$ is called \emph{Lagrangian} if the GZ fiber $\Phi_\lambda^{-1}(\mathbf{u})$ is Lagrangian for some point $\mathbf{u}$ (and hence each point) in the relative interior $\mathring{f}$. In $\mcal{O}_\lambda \simeq \mathcal{F}\ell(1,n;n+1)$, there are $n$ Lagrangian faces consisting of the improper face $f_0 \coloneqq \Delta_\lambda$ and $(n-1)$ Lagrangian proper faces $f_1, f_2, \cdots, f_{n-1}$ where $f_j$ is defined by
\begin{equation}\label{equ_fjface}
f_j \coloneqq \left\{ \mathbf{u} \in \Delta_\lambda \mid \lambda_2 = u_{1,j+1} = u_{1,j} = \cdots = u_{1,2} = u_{1,1} = u_{2,1} = \cdots = u_{j,1} =  u_{j+1,1}  \right\} . 
\end{equation}
If $\mathbf{u}$ is in $\mathring{f}_j$, then $\Phi_{\lambda}^{-1}(\mathbf{u})$ is a Lagrangian submanifold diffeomorphic to 
${S}^{2j + 1} \times T^{\dim f}$ where $T^0$ is assumed to be a point.

\begin{theorem}\label{CKO_topologyoffiber}
Each Gelfand--Zeitlin fiber $\Phi_{\lambda}^{-1}(\mathbf{u})$ is an isotropic submanifold diffeomorphic to 
\begin{equation}\label{equ_producttypesf}
\Phi_{\lambda}^{-1}(\mathbf{u}) \simeq \mcal{S}_f \times (S^1)^{\dim f}
\end{equation}
where $\mcal{S}_f$ is determined by~\eqref{equ_topologyofsf}.
\end{theorem}

\begin{proof}[Sketch of Proof]
By \cite[Theorem B]{ChoKimOhLag}, the torus of $\Phi_\lambda^{-1}(\mathbf{u})$ can be factored out so that
$$
\Phi_\lambda^{-1}(\mathbf{u}) \simeq \mcal{S} \times (S^1)^{\dim f}
$$
where $\mathcal{S}$ is the total space of an iterated bundle. Since the width of the ladder diagram $\Gamma_n$ is one, \emph{at most one} $L$-block consisting of multiple unit blocks can be inserted in $\Gamma \in \scr{P}(\Gamma_n)$ corresponding to $f$ via the map $\Psi$ in Theorem~\ref{theorem_AnChoKim}. Moreover, the only possible position of inserting such $L$-blocks is around the corner at the origin. It implies that the iterated bundle for $\mathcal{S}$ consists of a \emph{single} non-trivial stage, which is exactly $\mcal{S}_f$.
\end{proof}

\section{Main results}

The aim of this section is to present the main results of this paper. 

From now on, we take our choice $\lambda = (\lambda_1, \lambda_2, \lambda_3)$ in~\eqref{equ_lambdachoice} as  
\begin{equation}\label{equ_lambda123}
\lambda_1 = n(n-1), \lambda_2 = 0, \lambda_3 = - n(n-1)
\end{equation}
for the coadjoint orbit $\mcal{O}_\lambda \simeq \mathcal{F}\ell (1,n;n+1)$. Recall that this choice determines the K\"{a}hler form and the Gelfand--Zeitlin polytope and system as in~\eqref{equ_multipleprodfsfs},~\eqref{equation_GZ-pattern}, and~\eqref{equ_GZsystemcomp} accordingly. In particular, the KKS form $\omega_\lambda$ becomes \emph{monotone}, that is,  
$$
c_1 (T \mcal{O}_\lambda) = \nu \cdot [\omega_\lambda] \quad \mbox{ in ${H}_2 (\mcal{O}_\lambda)$ for some $\nu > 0$.}
$$

Recall that a Lagrangian submanifold $L$ of a symplectic manifold $(X, \omega)$ is called \emph{monotone} if for each relative homotopy class $\beta \in \pi_2(X, L)$, the Maslov index of $\beta$ and the symplectic area of $\beta$ are positively proportional. The monotone symplectic manifold $\mcal{O}_\lambda$ carries monotone GZ Lagrangian fibers. Indeed, each Lagrangian face of $\Delta_\lambda$ has a (unique) point in its relative interior whose GZ fiber is monotone by \cite[Theorem B]{ChoKimMONO}. In this circumstance, the classification result leads to the following theorem.

\begin{theorem}[\cite{ChoKimMONO}]\label{theroem_classificationmonotone}
For each $j = 0, 1, 2, \cdots, n-1$, let $\mathbf{u}$ be in the relative interior of the Lagrangian face $f_j$ in~\eqref{equ_fjface}. Then the Gelfand--Zeitlin fiber at $\mathbf{u}$ is monotone if and only if $\mathbf{u}$ is located at
$$
\begin{cases}
u_{1,k} = - u_{k,1} = 0 &\mbox{ for $k = 1, \cdots, j+1$} \\
u_{1,k} = - u_{k,1} = (n - 1)(k - 1) &\mbox{ for $k = j+2, \cdots, n$}.
\end{cases}
$$
\end{theorem}

In particular, for the improper Lagrangian face $f_0$, the monotone Lagrangian GZ torus fiber occurs at the center of $\Delta_\lambda$
$$
u_{1,k} = - u_{k,1} = (n - 1)(k - 1) \mbox{ for $k = 1, \dots, n$}.
$$ 
The center of $\Delta_\lambda$ is denoted by ${\mathbf{u}}_0$. We also choose a point ${\mathbf{u}}_1$ located at
\begin{equation}\label{equ_locationofu1}
u_{1,1} = u_{2,1} = u_{1,2} = 0, u_{1,k} = - u_{k,1} = {n(k-2)} \quad \mbox{for $k = 3, \dots, n$.}
\end{equation}
The point is contained in the relative interior of the Lagrangian face $f_1$. By Theorem~\ref{theroem_classificationmonotone}, the GZ fiber $\Phi^{-1}(\mathbf{u}_1)$ is monotone when $n = 2$ , while the GZ fiber $\Phi^{-1}(\mathbf{u}_1)$ is non-monotone when $n \geq 3$. When $n = 2$, the monotone GZ fiber $\Phi^{-1}(\mathbf{u}_1) \simeq S^3$ was shown to be non-displaceable in \cite{ChoKimOhLGbulk}. The main result in this paper proves non-displaceability of the \emph{non-monotone} and \emph{non-torus} GZ fiber located at $\mathbf{u}_1$ in $\mcal{O}_\lambda$ for $n \geq 3$.

\begin{theorem}\label{theorem_main}
For $n \geq 3$, consider the coadjoint orbit $\mcal{O}_\lambda \simeq \mathcal{F}\ell(1,n;n+1)$ equipped with the \emph{monotone} Kirillov--Kostant--Souriau symplectic form $\omega_\lambda$ from the choice~\eqref{equ_lambda123}. Then the Gelfand--Zeitlin fiber located at $\mathbf{u}_1$ in~\eqref{equ_locationofu1} is a \emph{non-monotone} and \emph{non-displaceable} Lagrangian submanifold diffeomorphic to ${S}^3 \times {T}^{2n-4}$.
\end{theorem}

To show that $\Phi^{-1}(\mathbf{u}_1)$ is non-displaceable, we now take into account the line segment $I_n(t)$ connecting ${\bf{u}}_0$ and ${\bf{u}}_1$. Namely,  
\begin{equation}\label{equ_Int}
I_n(t) \coloneqq \left\{ {\bf{u}} \in \Delta_\lambda \mid \mathbf{u} = (1-t) \mathbf{u}_0 + t \mathbf{u}_1, \, (0 \leq t \leq 1) \right\}.
\end{equation}
Later on, we will show that the GZ fiber over each point in $I_n(t)$ $(0 < t \leq 1)$ is non-displaceable.

\begin{theorem}\label{theorem_GZtorusfibernon}
For a choice $\lambda$ in~\eqref{equ_lambda123}, the Gelfand--Zeitlin fiber over each point in the line segment $I_n(t)$ $(0 < t \leq 1)$ is non-displaceable. 
\end{theorem}

Theorem~\ref{theorem_GZtorusfibernon} and the following lemma lead to Theorem~\ref{theorem_main} as a corollary because the non-torus fiber $\Phi^{-1}(\mathbf{u}_1)$ can be realized as a ``limit" of non-displaceable torus fibers.

\begin{lemma}[Proposition 2.10 in \cite{ChoKimOhLGbulk}]\label{lemma_limitnondisplace}
Suppose that $\Phi \colon X \to \R^N$ is a proper completely integrable system. If there is a sequence $\{ \mathbf{u}_j \in \Phi(X) \}_{j \in \mathbb{N}}$ of positions such that
\begin{enumerate}
\item each fiber $\Phi^{-1}(\mathbf{u}_j)$ is non-displaceable,
\item $\mathbf{u}_j$ converges to $\mathbf{u}_\infty \in \Phi(X)$,
\end{enumerate}
then the fiber $\Phi^{-1}(\mathbf{u}_\infty)$ is also non-displaceable.
\end{lemma}

\begin{proof}[Proof of Theorem~\ref{theorem_main}]
By Theorem~\ref{theorem_GZtorusfibernon}, the fiber over each point in the line segment $I_n(t)$ $(0 < t \leq 1)$ is non-displaceable. Lemma~\ref{lemma_limitnondisplace} yields non-displaceability of the fiber at $I_n(0)$.
\end{proof} 

\begin{example}
As depicted in Figure~\ref{Lagfacesfl134} (a), there are three Lagrangian faces of $\Delta_\lambda$ for $\lambda = (6, 0, 0, -6)$. Theorem~\ref{theorem_main} tells us that $S^3 \times T^2$ located at $(u_{1,3} = 3, u_{1,2} = u_{1,1} = u_{2,1} = 0, u_{3,1} = -3)$ is non-displaceable.

\begin{figure}[ht]
	\scalebox{1}{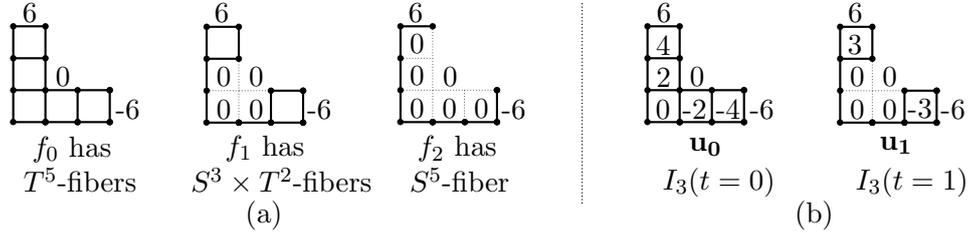}
	\caption{\label{Lagfacesfl134} In (a), there are three Lagrangian faces $f_0, f_1$, and $f_2$ of $\Delta_\lambda$ with $\lambda = (6, 0, 0, -6)$. 
	In (b), ${\bf{u}}_0$ and ${\bf{u}}_1$ are the end points of $I_n(t)$.}	
\end{figure}
\end{example}

For the remaining sections, we prove Theorem~\ref{theorem_GZtorusfibernon} by deforming Lagrangian Floer theory by ambient cycles. In Section~\ref{section_pseudocycles}, we describe cycles that will be employed for deformations. In Section~\ref{Sec_bulkdeformation}, we find a bulk-deformation making the deformed Floer cohomology non-vanishing.

\section{Pseudocycles in coadjoint orbits}\label{section_pseudocycles}

To prove Theorem~\ref{theorem_GZtorusfibernon}, we need to take a cycle of the coadjoint orbit $\mcal{O}_\lambda$ to do a bulk-deformation of Lagrangian Floer theory. In our situation, some vanishing cycles occur through the degeneration $\psi_{1,0}$ in~\eqref{equ_commutingdiagramgrad}. To incorporate those vanishing cycles, a pseudocycle is a suitable notion. 

We begin by recalling a notion of pseudocycles in \cite{ZingerPseudo}. For a continuous map $\varphi \colon M \to X$ between two topological spaces, the \emph{boundary} of $\varphi$ is defined by
\begin{equation}\label{equ_Omegavarphi}
\Omega_\varphi = \bigcap_{{K \subset M \text{ cpt}}} \overline{\varphi(M-K)} 
\end{equation}
where the intersection runs over all compact subsets of $M$. 
\begin{definition}\label{def_pseudocyclesdef}
For a smooth manifold $X$, a smooth map $\varphi \colon M \to X$ is called a \emph{$k$-pseudocycle} if
\begin{enumerate}
\item the domain $M$ is an oriented manifold of dimension $k$, 
\item the image $\varphi(M)$ is pre-compact, and 
\item there exists an open neighborhood $U$ of the boundary $\Omega_\varphi$ in~\eqref{equ_Omegavarphi} satisfying
\begin{equation}\label{equ_HellU}
H_\ell (U; \Z) = 0 \mbox{ for all $\ell > k-2$}.
\end{equation}
\end{enumerate}
\end{definition}

\begin{remark}
A subset $Z$ of a smooth manifold $X$ is said to have \emph{dimension at most} $k$ if there exists a manifold $Y$ of dimension $k$ and a smooth map $\phi \colon Y \to X$ such that $Z \subset \mathrm{Im}(\phi)$. In \cite{ZingerPseudo}, a $k$-pseudocycle is defined by the condition $(1), (2)$, and 
\begin{enumerate}[label=($\theenumi^\prime$)]
\setcounter{enumi}{2}
\item the dimension of $\Omega_\varphi$ is at most $k-2$.
\end{enumerate}
As mentioned in Section 1.2 therein, $(3^\prime)$ implies $(3)$ and the weakened condition $(3)$ is enough to ensure that a pseudocycle gives rise to an integral cycle.
\end{remark}

\begin{lemma}[Lemma 3.5 in \cite{ZingerPseudo}]\label{theorempsudocycle}
Every $k$-pseudocycle represents an integral $k$-cycle of $X$.
\end{lemma}

For the reader's convenience, we include the proof of Lemma~\ref{theorempsudocycle} presented in \cite{ZingerPseudo}.

\begin{proof}
Suppose that $\varphi \colon M \to X$ is a $k$-pseudocycle and there exists an open neighborhood $U$ of $\Omega_\varphi$ in $X$ satisfying~\eqref{equ_HellU}. Then, by the long exact sequence of homology groups of the pair $(X, U)$, the induced homomorphism given by the inclusion is an isomorphism 
$$
H_k (X, U; \Z) \simeq H_k (X; \Z).
$$ 

Since the closure $\overline{\varphi(M)}$ is compact in $X$, the complement $K \coloneqq M - \varphi^{-1}(U)$ is a compact subset of $M$ by the definition of $\Omega_\varphi$. 
We take an open neighborhood $V$ of $K$ in $M$ such that the closure $\overline{V}$ is a compact manifold with boundary. The closure $\overline{V}$ carries an inherited orientation from $M$ and the orientation $[\overline{V}]$ in $H_k (\overline{V}, \partial \overline{V}; \Z)$. Then we have an integral class $\varphi_* ([\overline{V}])$ in $H_k (X, U; \Z) \simeq H_k (X; \Z)$. 
\end{proof}

We return back to the Gelfand--Zeitlin case. To deform Floer cohomology later on, we shall employ the cycle corresponding to a pseudocycle. Keeping the toric degeneration in Theorem~\ref{theorem_toricdegenerations} in mind, consider the following toric divisors of the toric variety $\mcal{X}_0 \coloneqq \mcal{X} \cap V(t =0) \colon$
\begin{equation}\label{equ_toricdivisors}
\scr{D}_0 \coloneqq V(p_{n+1} ) \mbox{ and } \underline{\scr{D}}_0 \coloneqq V(p_{\underline{n+1}} ).
\end{equation}
The above divisors $\scr{D}_0$ and $\underline{\scr{D}}_0$ are respectively the inverse image of $f$ and $\underline{f}$ under the moment map $\Phi_0$ in ~\eqref{equ_coorGC} where 
\begin{equation}
f = \{ \mathbf{u} \in \Delta_\lambda \mid u_{1,n} = n(n-1) \} \quad \mbox{and} \quad \underline{f} =  \{ \mathbf{u} \in \Delta_\lambda \mid u_{n,1} = - n(n-1) \}.
\end{equation}
Note that $\dim f = \dim \underline{f} = \dim \Delta_\lambda - 1 = (2n - 1) - 1 = 2n - 2$.

\begin{proposition} Let $\scr{D}_0$ and $\underline{\scr{D}}_0$ be the toric divisors of the toric variety $\mcal{X}_0$ defined ~\eqref{equ_toricdivisors}.
\begin{enumerate}
\item When $n = 2$, the divisors $\scr{D}_0$ and $\underline{\scr{D}}_0$ are smooth.
\item When $n \geq 3$, the divisors $\scr{D}_0$ and $\underline{\scr{D}}_0$ are \emph{singular} toric varieties. Its singular locus is the toric subvariety of complex codimension three of $\scr{D}_0$ (resp. $\underline{\scr{D}}_0$) corresponding to the face given by
$$
\{ u_{1,1} = u_{1,2} = u_{2,1} = 0 \} \cap f \quad (\mbox{resp. $\{ u_{1,1} = u_{1,2} = u_{2,1} = 0 \} \cap \underline{f}$}). 
$$
Thus, the singular locus $\mathrm{Sing}(\scr{D}_0)$ of $\scr{D}_0$ is equal to the intersection $\mathrm{Sing}(\mcal{X}_0)  \cap \scr{D}_0$.
\end{enumerate}
\end{proposition}

\begin{proof}
As the involution $\iota \colon \mcal{X}_0 \to \mcal{X}_{{0}}$ determined by $p_j \leftrightarrow p_{\underline{j}}$ for $j = 1, 2, \cdots, n+1$ swaps $\scr{D}_0$ and $\underline{\scr{D}}_0$, it suffices to show the statement for $\scr{D}_0$. When $n = 2$, the facet given by $u_{1,2} = 2$ has four vertices at $(u_{1,2}, u_{1,1}, u_{2,1}) = (2, 0, 0), (2,2,0), (2,2,-2), (2,-2,-2)$ and hence $\scr{D}_0$ is isomorphic to the Hirzebruch surface $\mathbb{F}_1$ of degree one. In particular, the divisor $\scr{D}_0$ is smooth.

Suppose that $n \geq 3$. Under the isomorphism $(\CP^{n} \times \CP^n) \cap V(p_{n+1}) \simeq \CP^{n-1} \times \CP^n$ via the projection, the intersection $\mcal{X}_0 \, \cap V(p_{n+1})$ is then defined by $p_1 {p}_{\underline{1}} - p_2 {p}_{\underline{2}}$ in $\CP^{n-1} \times \CP^n$. This variety is a singular variety whose singular locus is given by $p_1 = p_{\underline{1}} = p_2 = p_{\underline{2}} = 0$. By the explicit expression of the moment map $\Phi_0$ in~\eqref{equ_coorGC}, the singular locus is equal to the inverse image of the face defined by $u_{1,1} = u_{1,2} = u_{2,1} = 0$. In other words, 
$$
\mathrm{Sing}(\scr{D}_0) \coloneqq \Phi_0^{-1} (\{ u_{1,1} = u_{1,2} = u_{2,1} = 0 \}) \cap \scr{D}_0
$$
as desired.
\end{proof}

When $n \geq 3$, let us take 
$$
M \coloneqq  \scr{D}_0 \backslash \mathrm{Sing}(\scr{D}_0) = \scr{D}_0 \backslash \mathrm{Sing}(\mcal{X}_0) \mbox{ and } \underline{M} \coloneqq \underline{\scr{D}}_0 \backslash \mathrm{Sing}(\scr{D}_0) = \underline{\scr{D}}_0 \backslash \mathrm{Sing}(\mcal{X}_0).
$$ 
Recall from Theorem~\ref{theorem_toricdegenerations} that the map $\psi_{1,0}$ is a symplectomorphism on the smooth locus $\psi_{1,0}^{-1} ( \mcal{X}_0 \backslash \mathrm{Sing}(\mcal{X}_0))$. We then have smooth maps
\begin{equation}\label{equ_fpsi10}
\begin{split}
&\varphi \coloneqq \psi^{-1}_{1,0} \big{|}_{M} \colon M \to X= \mcal{F}\ell (1, n ; n+1) \mbox{ and } 
\\ &\underline{\varphi} \coloneqq \psi^{-1}_{1,0} \big{|}_{\underline{M}} \colon \underline{M} \to X= \mcal{F}\ell (1, n ; n+1). 
\end{split}
\end{equation}
The following proposition claims that  $\varphi$ and $\underline{\varphi}$ are $(4n-4)$-pseudocycles. 

\begin{proposition}\label{proposition_pseudocycle}
The maps $\varphi$ and $\underline{\varphi}$ in~\eqref{equ_fpsi10} are $(4n-4)$-pseudocycles.
\end{proposition}

We shall only prove that the map $\varphi$ is a $(4n-4)$-pseudocycle as the other map can be similarly dealt with. To convince ourselves of the statement of Proposition~\ref{proposition_pseudocycle}, we first do some dimension counting. Since $M$ is an open  subset of a divisor, 
$$
\dim M =  \dim_\R \mcal{F}\ell (1, n;n+1) - 2 =  2 \cdot (2n-1) - 2 = 4n - 4.
$$
The {boundary} $\Omega_\varphi$ is contained in $\Phi^{-1}_\lambda (g)$ where $g$ is the stratum of codimension $3$ given by
$$
g \coloneqq \{ u_{1,1} = u_{1,2} = u_{2,1} = 0 \} \cap f
$$
because for each $\epsilon > 0$, a compact subset of $M$ can be chosen as $K_\epsilon \coloneqq \Phi^{-1}_0 (\Delta_\lambda \backslash U_\epsilon) \cap \scr{D}_0$ where $U_\epsilon$ is an $\epsilon$-neighborhood of $g$, that is, $U_\epsilon = \{ \mathbf{u} \in \Delta_\lambda \mid d_{\mathrm{Euc}} (\mathbf{u}, g) < \epsilon \}$. Since each fiber in the relative interior $\mathring{g}$ of $g$ has a vanishing $S^3$-factor according to Theorem~\ref{CKO_topologyoffiber}, 
$$
\dim_\mathbb{R} (\Phi^{-1}_\lambda (g)) = 
2 \cdot \dim g + \dim S^3 = 2 \cdot (\dim f - 3) + 3 = 4n - 7.
$$
Therefore, the boundary has at least codimension $3$.

We now start a proof of Proposition~\ref{proposition_pseudocycle}. The complex submanifold $M = \scr{D}_0 \backslash \mathrm{Sing}(\scr{D}_0)$ carries a canonical orientation and hence the condition $(1)$ holds. Since the range $X = \mcal{F}\ell(1,n;n+1)$ is compact, the condition $(2)$ is automatically satisfied. It remains to provide an open neighborhood $\mathscr{U}$ of $\Omega_\varphi$ such that ${H}_\ell (\mathscr{U}; \Z) = 0 $ for all $\ell > \dim_\R M - 2 = (4n - 4) - 2$. 

The polytope $\Delta_\lambda$ is stratified by its faces so that it can be expressed as the disjoint union of relative interiors of its faces. We label the relative interior of each face of $\Delta_\lambda$ with a double index as follows. The first component of the double index keeps track of the dimension of the face, that is, $\dim f_{i,\bullet} = i$. The second component is chosen by fixing an ordering on the set of the relative interiors of $i$-dimensional faces, $f_{i,1}, f_{i,2}, \cdots, f_{i, \kappa^\prime_i}$ where $\kappa^\prime_i$ is the number of $i$-dimensional faces of $\Delta_\lambda$. Note that $\kappa^\prime_{2n-1} = 1$ and the interior of $\Delta_\lambda$ is equal to $f_{2n-1,1}$. The polytope $\Delta_\lambda$ can be expressed as the disjoint union
\begin{equation}\label{equ_stratifi}
\Delta_\lambda = \bigcup_{i=0}^{2n-1} \bigcup_{j=1}^{\kappa^\prime_i} f_{i,j}.
\end{equation}
The set of faces of $\Delta_\lambda$ possesses a partial ordering $\subset$ given by the set-theoretical inclusion. Recall that this partially ordered set can be described via the order-preserving correspondence $\Psi$ in Theorem~\ref{theorem_AnChoKim}. 

For each point $\mathbf{u} \in f_{i,j}$, we take an open ball ${U}$ around $\mathbf{u}$ in the subspace $f_{i,j}$. 
For each point $\mathbf{u}^\prime \in {U}$, the line segment connecting $\mathbf{u}$ and $\mathbf{u}^\prime$ within $f_{i,j}$ determines a diffeomorphism $\rho_{\mathbf{u}^\prime} \colon \Phi^{-1}_\lambda (\mathbf{u}^\prime) \to \Phi^{-1}_\lambda (\mathbf{u})$, see \cite[Lemma 6.13]{ChoKimSO} for the construction. As the diffeomorphisms $\rho_\bullet$ depend continuously on $\mathbf{u}^\prime$, we then have a local trivialization
$$
\Phi_\lambda^{-1}({U}) \simeq {U}  \times \Phi_\lambda^{-1}(\mathbf{u}) \quad \quad x \mapsto \left(\Phi_\lambda(x), \rho_{\Phi_\lambda(x)}(x) \right)
$$
so that the map $\Phi_\lambda \colon \Phi^{-1}_\lambda(f_{i,j}) \to f_{i,j}$ is a fiber bundle. Since $f_{i,j}$ is contractible, the fiber bundle
\begin{equation}\label{equ_fiberbundleij}
\Phi_\lambda \colon \Phi^{-1}_\lambda(f_{i,j}) \to f_{i,j} \simeq \R^i
\end{equation}
can be trivialized. In other words, the Gelfand--Zeitlin system $\Phi_\lambda$ over each stratum of $\Delta_\lambda$ is a trivial fiber bundle. Set $\mathscr{F}_{i,j} \coloneqq \Phi^{-1}_\lambda(f_{i,j})$. In particular, we have
\begin{equation}\label{equ_inverseimage}
\mathscr{F}_{i,j}  \simeq \R^{i} \times \Phi^{-1}_\lambda (\mathbf{u}_{i,j})
\end{equation}
for any point $\mathbf{u}_{i,j}$ of $f_{i,j}$.

The stratification in~\eqref{equ_stratifi} induces that of the face $g$, say
$$
g = \bigcup_{i=0}^{2n-5} \bigcup_{j=1}^{\kappa_i} g_{i,j}
$$
where $\{g_{i,1}, g_{i,2}, \cdots, g_{i, \kappa_i}\}$ is the set of the relative interiors of $i$-dimensional faces of $g$ and $\kappa_i$ is the number of $i$-dimensional faces of $g$. We consider the lexicographic partial ordering $\preceq$ on the set of double indices on the set of (relative interiors of) faces of $g$, that is, 
$$
(r,s) \preceq (i,j) \mbox{ if and only }
r \leq i \mbox{ or } (r = i \mbox{ and } s < j). 
$$
To obtain a desired open neighborhood $\mathscr{U}$ of $\Phi^{-1}_\lambda (g) (\supset  \Omega_\varphi)$, we are planning to glue open neighborhoods $\mathscr{U}_{i,j}$'s of the inverse image of $g_{i,j}$'s inductively with respect to the partial ordering $\preceq$. By applying the Mayer--Vietoris sequence to the glued set, we shall verify the desired vanishing result on homology groups at the end. 

For a sufficiently small positive number $\epsilon > 0$, let us take an open neighborhood $U_{i,j}$ of $g_{i,j}$ 
\begin{equation}\label{equuij}
\begin{split}
U_{i,j} \coloneqq  \left\{ \mathbf{u} \in \Delta_\lambda \mid d_\mathrm{Eud} (\mathbf{u}, \mathbf{u}^\prime) <  \epsilon \mbox{ for some $\mathbf{u}^\prime \in g_{i,j}$} \right\} \backslash \left( \cup_{\clubsuit}  \, {{f_{r,s}}} \right)
\end{split}
\end{equation}
where the union $\cup_\clubsuit$ runs over all faces $f_{r,s}$ of $f$ such that $\overline{f_{r,s}}  \cap g_{i,j} = \emptyset$. Since $\cup_{\clubsuit} \, {f_{r,s}} = \cup_{\clubsuit} \, {\overline{f_{r,s}}}$, the set $U_{i,j}$ is an open set. We obtain an open neighborhood $\mathscr{U}_{i,j}$ of $\mathscr{G}_{i,j} \coloneqq \Phi^{-1}_\lambda (g_{i,j})$ defined by
\begin{equation}\label{equ_openneighborhoods}
\mathscr{U}_{i,j} \coloneqq \Phi^{-1}_\lambda \left( U_{i,j} \right).
\end{equation}

\begin{lemma}\label{lemma_homologygroupofuij} 
If $n \geq 3$, then for each face $g_{i,j}$ of the face $g$, the open neighborhood $\mathscr{U}_{i,j}$ of $\mathscr{G}_{i,j} \coloneqq \Phi^{-1}_\lambda (g_{i,j})$ in~\eqref{equ_openneighborhoods} satisfies
\begin{equation}\label{equ_huij0}
H_\ell (\mathscr{U}_{i,j} ; \Z) = 0 \mbox{ for all $\ell \geq \dim_\R \Phi_\lambda^{-1}(f) - 3 = (4n - 4) - 3 = 4n - 7$}. 
\end{equation}
\end{lemma}

\begin{proof}
We claim that the fiber $\Phi^{-1}_\lambda (\mathbf{u}_{i,j})$ at a point $\mathbf{u}_{i,j}$ of $g_{i,j}$ is a deformation retract of $\mathscr{U}_{i,j}$. First, the set $U_{i,j}$ in~\eqref{equuij} can be contracted to the point $\mathbf{u}_{i,j}$ strata-wisely as depicted in Figure~\ref{stratawise}~(a). For $\ell = 0, 1, 2, \cdots$, and $n-i$, we set
$$
U^{(\ell)}_{i,j} \coloneqq U_{i,j} \cap \left( g_{i,j} \cup \left(\cup_{\spadesuit_\ell} f_{r,s} \right)\right)
$$
where the union $\cup_{\spadesuit_\ell}$ runs over all faces $f_{r,s}$ of $f$ such that $\overline{f_{r,s}} \cap g_{i,j} \neq \emptyset$ and $\mathrm{codim} \, f_{r,s} \geq \ell$. Note that $U_{i,j} = U_{i,j}^{(0)}$. By collapsing the strata of codimension $\ell$, we obtain a deformation retraction of $U^{(\ell)}_{i,j}$ into $U^{(\ell+1)}_{i,j}$. Proceeding inductively, we end up obtaining the stratum $U^{(n-i)}_{i,j} = g_{i,j}$. As $g_{i,j} \simeq \R^{i}$, it can be contracted to the point $\mathbf{u}_{i,j}$. Next, we now lift them to the total space $\mathscr{U}_{i,j}$ of $\Phi_\lambda$. The deformation retraction of  $U^{(\ell)}_{i,j}$ to $U^{(\ell+1)}_{i,j}$ together with the trivialization of $\Phi_\lambda$ over each stratum of codimension $\ell$ leads to a deformation retraction of $\Phi_\lambda^{-1} (U^{(\ell)}_{i,j})$ to $\Phi_\lambda^{-1} (U^{(\ell+1)}_{i,j})$. Then the fiber $\Phi^{-1}_\lambda (\mathbf{u}_{i,j})$ is a deformation retract of $\mathscr{G}_{i,j}$ by~\eqref{equ_inverseimage}.
 
By Theorem~\ref{CKO_topologyoffiber}, $\Phi^{-1}_\lambda (\mathbf{u}_{i,j})$ is diffeomorphic to $\mathcal{S}_{i,j} \times T^i$ where $\mathcal{S}_{i,j}$ is either a single point or an odd dimensional sphere. Therefore, 
$$
H_\ell (\mathscr{U}_{i,j} ; \Z) \simeq H_\ell (\mcal{S}_{i,j} \times T^i ; \Z).
$$
Since every fiber is an isotropic submanifold by Theorem~\ref{CKO_topologyoffiber}, we have the inequality$\colon$
\begin{equation}\label{dimsijt}
\dim (\mathcal{S}_{i,j}) + i \leq 2n-1.
\end{equation}
If $n > 3$, then $\dim \mcal{S}_{i,j} \times T^i \leq 2n - 1 < 4n - 7$ by~\eqref{dimsijt} and hence~\eqref{equ_huij0} follows. If $n = 3$, then the face $f$ is contained in the hyperplane given by $u_{1,3} = 6$. Also, the face $g$ is a line segment and can be expressed as the union of three faces
$$
g_{0,1} = \{ \mathbf{u} \in f \mid u_{3,1} = 0 \}, \,g_{0,2} = \{  \mathbf{u} \in f \mid u_{3,1} = -6 \}, \, g_{1,1} = \{  \mathbf{u} \in f \mid -6 < u_{3,1} < 0 \}.
$$
The corresponding factors $\mathcal{S}_{i,j}$ for $\Phi^{-1}_\lambda (\mathbf{u}_{i,j})$ in Theorem~\ref{CKO_topologyoffiber} are
\begin{itemize}
\item $\mathcal{S}_{0,1}$ is a point and 
\item $\mathcal{S}_{0,2} \simeq \mathcal{S}_{1,1} \simeq S^3$. 
\end{itemize}
Therefore,~\eqref{equ_huij0} holds for the case where $n = 3$ as well. 
\begin{figure}[ht]
	\scalebox{1}{
\begingroup%
  \makeatletter%
  \providecommand\color[2][]{%
    \errmessage{(Inkscape) Color is used for the text in Inkscape, but the package 'color.sty' is not loaded}%
    \renewcommand\color[2][]{}%
  }%
  \providecommand\transparent[1]{%
    \errmessage{(Inkscape) Transparency is used (non-zero) for the text in Inkscape, but the package 'transparent.sty' is not loaded}%
    \renewcommand\transparent[1]{}%
  }%
  \providecommand\rotatebox[2]{#2}%
  \newcommand*\fsize{\dimexpr\f@size pt\relax}%
  \newcommand*\lineheight[1]{\fontsize{\fsize}{#1\fsize}\selectfont}%
  \ifx\svgwidth\undefined%
    \setlength{\unitlength}{976.86626503bp}%
    \ifx\svgscale\undefined%
      \relax%
    \else%
      \setlength{\unitlength}{\unitlength * \real{\svgscale}}%
    \fi%
  \else%
    \setlength{\unitlength}{\svgwidth}%
  \fi%
  \global\let\svgwidth\undefined%
  \global\let\svgscale\undefined%
  \makeatother%
  \begin{picture}(1,0.0928991)%
    \lineheight{1}%
    \setlength\tabcolsep{0pt}%
    \put(0,0){\includegraphics[width=\unitlength,page=1]{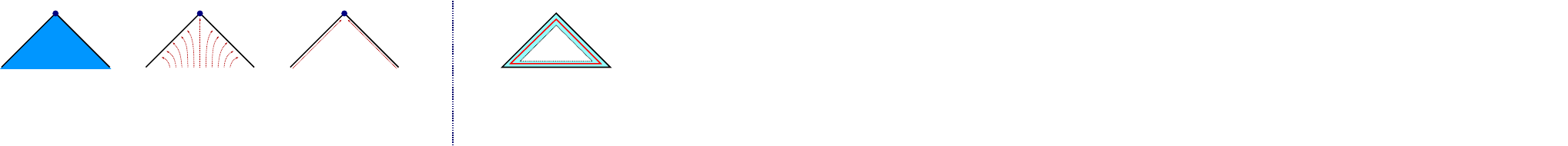}}%
    \put(0.34872259,0.02956402){\color[rgb]{0,0,0}\makebox(0,0)[lt]{\lineheight{1.25}\smash{\begin{tabular}[t]{l}$R_{i,j}$\end{tabular}}}}%
    \put(0,0){\includegraphics[width=\unitlength,page=2]{fig_def.pdf}}%
    \put(0.34872259,0.00960223){\color[rgb]{0,0,0}\makebox(0,0)[lt]{\lineheight{1.25}\smash{\begin{tabular}[t]{l}$U^\prec_{i,j} \cap U_{i,j} \cap g_{i,j}$\end{tabular}}}}%
    \put(0,0){\includegraphics[width=\unitlength,page=3]{fig_def.pdf}}%
    \put(0.02779841,0.02035088){\color[rgb]{0,0,0}\makebox(0,0)[lt]{\lineheight{1.25}\smash{\begin{tabular}[t]{l}(a) Strata-wise deformation retraction\end{tabular}}}}%
    \put(0.29571275,0.02006763){\color[rgb]{0,0,0}\makebox(0,0)[lt]{\lineheight{1.25}\smash{\begin{tabular}[t]{l}(b)\end{tabular}}}}%
  \end{picture}%
\endgroup%
}
	\caption{\label{stratawise} Strata-wise deformation retraction and choice of $R_{i,j}$}	
\end{figure}
\end{proof}

Taking a sufficiently small positive number $\epsilon$ for $U_{i,j}$'s, we may assume that there is an open ball $V_{i,j}$ around the center $\mathbf{u}_{i,j}$ of $g_{i,j}$ such that $V_{i,j}$ and $U_{i,j}^{\prec}$ are disjoint. We inductively set
\begin{itemize}
\item $U_{1,1}^{\prec} \coloneqq \bigcup_{j=1}^{\kappa_0} U_{0,j} \mbox{ and } \mathscr{U}_{1,1}^{\prec} \coloneqq \bigcup_{j=1}^{\kappa_0} \mathscr{U}_{0,j},
$
\item $
U_{i,j+1}^{\prec} \coloneqq U_{i,j}^{\prec} \cup U_{i,j} \mbox{ and } \mathscr{U}_{i,j+1}^{\prec} \coloneqq \mathscr{U}_{i,j}^{\prec} \cup \mathscr{U}_{i,j}
$ if $j < \kappa_i$,
\item $
U_{i+1,1}^{\prec} \coloneqq U_{i,j}^{\prec} \cup U_{i+1,1} \mbox{ and } \mathscr{U}_{i+1,1}^{\prec} \coloneqq \mathscr{U}_{i,j}^{\prec} \cup \mathscr{U}_{i+1,1}
$
if $j = \kappa_i$.
\end{itemize}
To compute homology groups of $\mathscr{U}_{i,j+1}^{\prec}$, let us compute homology groups of their intersection.

\begin{lemma}\label{lemma_Hofinter}
For every $n \geq 3$, we have
\begin{equation}\label{equ_vanishinginter}
H_\ell \left(\mathscr{U}_{i,j}^{\prec} \cap \mathscr{U}_{i,j}; \Z \right) = 0 \mbox{ for all $\ell \geq 4n - 7$}.
\end{equation}
\end{lemma}

\begin{proof}
Under the identification $g_{i,j} \simeq \R^{i}$ via a homeomorphism, we may take an $(i-1)$-dimensional sphere ${R}_{i,j}$ in~\eqref{equuij} with sufficiently large radius such that ${R}_{i,j}$ is contained in $U_{i,j}^\prec \cap U_{i,j} \cap g_{i,j}$, see Figure~\ref{stratawise}~(b). Let $\scr{R}_{i,j} \coloneqq \Phi_\lambda^{-1} (R_{i,j})$. Because $\epsilon$ is sufficiently small, the deformation retraction of $\mathscr{U}_{i,j}$ to $\mathscr{G}_{i,j}$ in Lemma~\ref{lemma_homologygroupofuij} induces that of $\mathscr{U}_{i,j}^{\prec} \cap \mathscr{U}_{i,j}$ to $\mathscr{R}_{i,j}$. In particular, $\mathscr{U}_{i,j}^{\prec} \cap \mathscr{U}_{i,j}$ is homotopy equivalent to $\mathscr{R}_{i,j}$. The trivialization of $\Phi_\lambda$ over $g_{i,j}$ yields that
$$
\mathscr{R}_{i,j} \simeq {R}_{i,j} \times \Phi^{-1}_\lambda (\mathbf{u}_{i,j}) \simeq S^{i-1} \times \mcal{S}_{i,j} \times T^i
$$
where the corresponding factor $\mathcal{S}_{i,j}$ for $\Phi^{-1}_\lambda (\mathbf{u}_{i,j})$ is in Theorem~\ref{CKO_topologyoffiber}.  Then
$$
\dim (S^{i-1} \times (\mcal{S}_{i,j} \times T^i) ) \leq (i - 1) + (2n - 1) \leq (2n - 5 - 1) + (2n - 1).
$$
Here, the first inequality follows from~\eqref{dimsijt} and the second inequality follows from the dimension count $i = \dim g_{i,j} \leq \dim g = \dim f - 3 = (2n - 2) - 3$. Therefore,~\eqref{equ_vanishinginter} follows. 
\end{proof}

\begin{proposition}\label{proposition_homologyvanishing}
For every $n \geq 3$ and each double index $(i,j)$, we have
\begin{equation}\label{equ_hellij}
H_\ell \left(\mathscr{U}_{i,j}^{\prec}; \Z \right) = 0 \mbox{ for all $\ell \geq 4n - 6$}.
\end{equation}
\end{proposition}

\begin{proof}
We use the mathematical induction on $\prec$. For the base step, observe that $\mathscr{U}_{1,1}^{\prec}$ is homotopy equivalent to $\bigcup_{j=1}^{\kappa_0} \mathcal{S}_{0, j}$ where $\mathcal{S}_{0, j}$ is the homeomorphic type of the fiber at the vertex $g_{0,j}$. Thus, $H_\ell (\mathscr{U}_{i,j}^{\prec}; \Z ) = 0$ for all $\ell \geq 4n - 6$. 

As an induction hypothesis, assume that $\mathscr{U}^{\prec}_{i,j}$ satisfies~\eqref{equ_hellij}. By the Mayer--Vietoris sequence to the pair $(\mathscr{U}^{\prec}_{i,j+1}, \mathscr{U}_{i,j+1})$ (or $(\mathscr{U}^{\prec}_{i+1,1}, \mathscr{U}_{i+1,j})$ if $j = \kappa_i$),~\eqref{equ_hellij} follows from Lemma~\ref{lemma_homologygroupofuij} and~\ref{lemma_Hofinter}.
\end{proof}

\begin{proof}[Proof of Proposition~\ref{proposition_pseudocycle}]
The union $\mathscr{U} \coloneqq \bigcup_{i=0}^{2n-5} \bigcup_{j=1}^{\kappa_i} \mathscr{U}^{\prec}_{i,j}$ is an open neighborhood containing $\Phi^{-1}_\lambda (g)$ satisfying the desired vanishing properties on homology groups by Proposition~\ref{proposition_homologyvanishing}.
\end{proof}

\section{Non-displaceable Lagrangian torus fibers}\label{Sec_bulkdeformation}

The goal of this section is to complete the proof of Theorem~\ref{theorem_main}. After reviewing the bulk-deformation of Lagrangian Floer theory, we find a bulk-parameter making the bulk-deformed Floer cohomology of the Gelfand--Zeitlin fiber over each point in the line segment $I_n(t)$ $(0 < t \leq 1)$ non-vanishing.

Let $\Lambda$ be the Novikov field defined by 
$$
\Lambda \coloneqq \left\{ \sum_{i=1}^\infty a_i T^{\mu_i} \mid a_i \in \C, \mu_i \in \R, \lim_{i \to \infty} \mu_i = \infty \right\}
$$
where $T$ is a formal parameter. For a nonzero element $x \in \Lambda \backslash \{0\}$, the valuation $\frak{v}_T (x)$ is defined by the minimal number $\mu \in \R$ such that the complex part of $T^{-\mu} x$ is non-zero. Let
\begin{itemize}
\item $\mathrm{U}(\Lambda) \coloneqq \frak{v}_T^{-1}(0)$ be the set of unitary elements of $\Lambda$,
\item $\Lambda_{0} \coloneqq \left\{ \sum_{i=1}^\infty a_i T^{\mu_i} \in \Lambda \mid \mu_i \geq 0 \right\} = \{0\} \cup \frak{v}_T^{-1}([0, \infty))$.
\item for $a \in \R$, $\Lambda_{>a} \coloneqq \frak{v}_T^{-1}((a, \infty))$.
\end{itemize}

Let $(X, \omega)$ be a symplectic manifold and $L$ a Lagrangian submanifold. By the work of Fukaya--Oh--Ohta--Ono \cite{FukayaCyclic, FOOOToric1, FOOOToric2}, one can associate the $A_\infty$-algebra 
$$
\left( \Omega(L),  \{ \frak{m}_k \colon \Omega(L)^{\otimes k} \to \Omega(L) \}_{k=0}^\infty \right)
$$ 
on the de Rham complex of $L$. Let $\mathcal{M}_{k+1, \ell}(L, \beta)$ be the moduli space of stable maps from a bordered Riemann surface of genus zero with $(k+1)$ boundary marked points $\{z_0, z_1, \cdots, z_k\}$ respecting the counter-clockwise orientation and $\ell$ interior marked points $\{z_1^+, \cdots, z_\ell^+\}$. The moduli space has the evaluation maps
$$
\begin{cases}
\mathrm{ev}_j^{\vphantom{+}} \colon \mathcal{M}_{k+1, \ell}(L, \beta) \to L, &\quad \varphi \mapsto \varphi (z_j), \\
\mathrm{ev}_j^+ \colon \mathcal{M}_{k+1, \ell}(L, \beta) \to L, &\quad \varphi \mapsto \varphi (z^+_j).
\end{cases}
$$
For each $k = 0, 1, 2, \cdots$ and $\beta \in \pi_2(X, L)$, the map $\frak{m}_{k, \beta}$ is constructed via the smooth correspondence$\colon$
\begin{equation}\label{smoothcorrespondence}
	\xymatrix{
                              & {\mcal{M}_{k+1, 0}(L, \beta)} \ar[dl]_{\mathbf{ev}_{\geq 1}} \ar[dr]^{\mathrm{ev}_0} &
      \\
 L^k & & L}
\end{equation}
where $\mathbf{ev}_{\geq 1} = (\mathrm{ev}_1, \mathrm{ev}_2, \cdots, \mathrm{ev}_k)$. Namely, letting $\mathrm{pr}_j \colon L^k \to L$ be the projection to the $j$-th factor, we define
$$
\mathfrak{m}_{k,\beta}(b_1, \cdots, b_k) = (\mathrm{ev}_{0})_! \left( \mathbf{ev}_{\geq 1}^* \left(\mathrm{pr}_1^* b_1 \wedge \cdots \wedge \mathrm{pr}_k^* b_k \right) \right).
$$
The structure map $\mathfrak{m}_k$ is defined by
$$
\mathfrak{m}_k = \sum_{\beta} \mathfrak{m}_{k, \beta} \cdot T^{\omega (\beta) / 2 \pi},
$$
see \cite[Section 16]{FOOOToric1} for details.

For cycles $b \in H^1(L; \Lambda_0)$ and $\frak{b} \in H^2(X; \Lambda_0)$, the structure map of the $A_\infty$-algebra can be deformed by adding extra marked points to the moduli space ${\mcal{M}_{k+1, 0}(L, \beta)}$ of stable maps and requiring to pass through $\frak{b}$ at the interior marked points and $b$ at the extra boundary marked points. The smooth correspondence gives rise to a deformation of the map $\mathfrak{m}_{k, \beta}$ and $\mathfrak{m}_{k}$.  The structure map for the resulting deformed $A_\infty$-algebra is denoted by $\mathfrak{m}_k^{\frak{b}, b}$. If the output of $\mathfrak{m}^{\frak{b}, b}_0 (1)$ is a multiple of the fundamental cycle $\mathrm[PD(L)]$, then the disk potential function is defined by
$$
\mathfrak{m}^{\frak{b}, b}_0 (1) = \mathfrak{m}^\frak{b} ({e^b}) = \frak{PO}^\frak{b}(b) \cdot \mathrm[PD(L)].
$$
In the case where $\frak{b}$ is a linear combination of ambient cycles of real codimension two that do not intersect the Lagrangian $L$, the divisor axiom can be applied. Together with the compatibility of forgetful maps at the boundary marked points, the bulk-deformed disk potential of $L$ can be expressed as the following form$\colon$
\begin{equation}\label{equ_pob}
\frak{PO}^\frak{b} (b) = \sum_{\beta} n_\beta \cdot \exp (\beta \cap \frak{b}) \cdot \exp (\partial \beta \cap b) \cdot T^{\omega(\beta) / 2 \pi}
\end{equation}
where $n_\beta$ is the degree of the evaluation map after choosing orientation and spin structure. The reader is referred to \cite{FOOOToric2} for details.

We now deal with the Gelfand--Zeitlin case. For each nonzero $s$, the map $\psi_{1,s}$ in~\eqref{equ_commutingdiagramgrad} is given by a Hamiltonian flow and hence a Lagrangian submanifold $L$ in $\mathcal{X}_1$ is non-displaceable if and only if the image $\psi_{1,s}(L)$ is non-displaceable in ${X}_s$. From now on, we assume that $s$ is sufficiently close to the zero. Let $L_n(t) \coloneqq \Phi_\lambda^{-1} \left( I_n(t) \right)$ be the fiber over the point $I_n(t)$ in~\eqref{equ_Int}. By abuse of notation, we shall omit $\psi_{1,s}$ for simplicity. For instance, the image $\psi_{1,s}( L_n(t))$ is denoted by $L_n(t)$.

We denote by $\scr{D}$ (resp. $\underline{\scr{D}}$) the cycle corresponding to a pseudo-cycle $\varphi$ (resp. $\underline{\varphi}$) in Lemma~\ref{theorempsudocycle} and Proposition~\ref{proposition_pseudocycle}. Again for simplicity, the pushforward $\psi_{1,s,*} (\scr{D})$ in $X_s$ is also denoted by  $\scr{D}$. Taking a bulk-parameter of the form
\begin{equation}\label{equ_bulkparam}
\frak{b} = \frak{c} \cdot \scr{D} + \underline{\frak{c}} \cdot \underline{\scr{D}} \quad \mbox{ ($\frak{c}$ and $\underline{\frak{c}} \in \Lambda_0$)},
\end{equation}
for each $1$-cochain $b$, the obstruction $\frak{m}^{\frak{b},b}_0(1)$ is a multiple of the fundamental cycle, see \cite[Section 5]{ChoKimOhLGbulk}. 

We want to compute the bulk-deformed disk potential $\frak{PO}^\frak{b} (b)$ of $L_n(t)$. Recall that each component $\Phi_{i,j}$ generates a Hamiltonian circle action on $L_n(t)$. Let $\vartheta_{i,j}$ be an oriented loop in $L_n(t)$. We then take a coordinate system 
\begin{equation}\label{equ_yij}
\{ y_{i,j} \coloneqq \exp \left( \vartheta_{i,j} \cap b \right) \in \mathrm{U}(\Lambda) \}
\end{equation}
on $H^1(L; \Lambda_0)$. The expression of $\frak{PO}^\frak{b} (b)$ restricted to $H^1(L; \Lambda_0)$ in terms of $y_{i,j}$'s is a Laurent series, which is denoted by $W_\mathfrak{b}(L_n(t), {\bf{y}})$ (or $W_\mathfrak{b}({\bf{y}})$ for simplicity). The series $W_\mathfrak{b}({\bf{y}})$ is also called the \emph{bulk-deformed disk potential function} of $L_n(t)$.

\begin{proposition} The bulk-deformed potential function of $L_n(t)$ is 
\begin{equation}\label{equ_potentialfl1nn1}
\begin{split}
W_\mathfrak{b}(L_n(t), {\bf{y}}) &= \left( \frac{y_{1,2}}{y_{1,1}} + y_{1,2} + \frac{y_{1,1}}{y_{2,1}} + \frac{1}{y_{2,1}} \right) T^{(n-1)(1-t)} \\
& + \left( c \cdot \frac{1}{y_{1,n}} + \frac{y_{1,n}}{y_{1,n-1}} + \cdots +\frac{y_{1,3}}{y_{1,2}} +  \underline{c} \cdot {y_{n,1}} + \frac{y_{n-1,1}}{y_{n,1}} + \cdots + \frac{y_{2,1}}{y_{3,1}} \right) T^{n-1+t}
\end{split}
\end{equation}
for $c \coloneqq \exp(\mathfrak{c})$ and $\underline{c} \coloneqq \exp(\mathfrak{\underline{c}}) \in \mathrm{U}(\Lambda)$.
\end{proposition} 

\begin{proof} 
We need to compute the expression~\eqref{equ_pob}. By the work of Nishinou--Nohara--Ueda \cite{NishinouNoharaUeda}, the effective classes of Maslov index two are classified and the counting invariants are computed. In particular, the disk potential function of $L_n(t)$ without bulk $\frak{b} = 0$ can be expressed as a Laurent polynomial with respect to the variables $y_{i,j}$ in~\eqref{equ_yij} as follows$\colon$
\begin{equation}\label{equ_potentialfl1nn2}
\begin{split}
W_0(L_n(t),{\bf{y}}) &= \left( \frac{y_{1,2}}{y_{1,1}} + y_{1,2} + \frac{y_{1,1}}{y_{2,1}} + \frac{1}{y_{2,1}} \right) T^{(n-1)(1-t)} \\
& + \left( \frac{1}{y_{1,n}} + \frac{ y_{1,n}}{y_{1,n-1}} + \cdots +\frac{y_{1,3}}{y_{1,2}} + {y_{n,1}} + \frac{y_{n-1,1}}{y_{n,1}} + \cdots + \frac{y_{2,1}}{y_{3,1}} \right) T^{n-1+t}.
\end{split}
\end{equation}

Turning on the bulk-parameter $\frak{b}$ in~\eqref{equ_bulkparam}, we compute the intersection numbers between $\scr{D}$ and disk classes $\beta$ of Maslov index two. For each effective class $\beta$ of Maslov index two, a holomorphic disk in $\beta$ is contained in the smooth loci $\psi_{s,0}^{-1}(\mathring{\mcal{X}}_0)$ where $\mathring{\mcal{X}}_0 \coloneqq \mcal{X}_0 \backslash \mathrm{Sing}(\mcal{X}_0)$ and $\scr{D}$ degenerates into $\scr{D}_0$ via the map $\psi_{s,0}$. Since the map is a diffeomorphism on $\psi_{s,0} \colon \psi_{s,0}^{-1}(\mathring{\mcal{X}_0}) \to \mathring{\mcal{X}}_0$, the intersection number can be computed at $\mcal{X}_0$. Since $\psi_{s,0,*} (\beta)$ is a basic disk and $\scr{D}$ is a toric divisor, the intersection number can be easily computed. The only basic disk class intersecting with $\scr{D}$ (resp. $\underline{\scr{D}}$) is a holomorphic disk emanated from the divisor corresponding to $f$ (resp. $\underline{f}$). Moreover, the intersection number is exactly one.  Combining it with~\eqref{equ_pob} and~\eqref{equ_potentialfl1nn2}, we obtain the formula~\eqref{equ_potentialfl1nn1}.
\end{proof}

\begin{proposition}\label{proposition_criticalpoints}
There exist bulk-parameters $c, \underline{c} \in \mathrm{U}(\Lambda)$ such that the bulk-deformed potential function $W_\frak{b}(L(t), \mathbf{y})$ in~\eqref{equ_potentialfl1nn1} admits a critical point each of which component is in $\mathrm{U}(\Lambda)$. 
\end{proposition}

\begin{proof}[Proof of Theorem~\ref{theorem_GZtorusfibernon}] 
By Proposition~\ref{proposition_criticalpoints}, the bulk-deformed Floer cohomology of $L(t)$ is isomorphic to the ordinary cohomology of $L(t)$. Consequently, the Lagrangian torus $L(t)$ is non-displaceable, see \cite[Section 8]{FOOOToric2}. Theorem~\ref{theorem_GZtorusfibernon} follows from Proposition~\ref{proposition_criticalpoints}. 
\end{proof}

Thus, it remains to prove Proposition~\ref{proposition_criticalpoints}. In \cite{ChoKimOhLGbulk}, the authors introduced the split leading term equation associated with a Lagrangian face $f$ to detect Lagrangian tori with non-vanishing bulk-deformed Floer cohomology located at a line segment joining a point of $f$ and the barycenter of the GZ polytope. A (complex) solution of the split leading term equation (over $\C$) will be the valuation zero part of a critical point of the bulk-deformed potential function. The strategy is to solve the split leading term equation first and then extend the solution to a critical point of the bulk-deformed potential function. Thus, Proposition~\ref{proposition_criticalpoints} is established. 

In our setting, we explain how the split leading term equation can be obtained from the ladder diagram. As an example, we explain the split leading term equation associated with $f_1$ in~\eqref{equ_fjface} by the example when $n = 3$. Consider the diagram corresponding to $f_1$ which has the $L$-shaped region at the left-bottom corner.   
The potential function is arranged as follows$\colon$
\begin{equation}\label{equ_potentialfl134}
W({\bf{y}})  = \left( \frac{y_{1,2}}{y_{1,1}} + y_{1,2} + \frac{y_{1,1}}{y_{2,1}} + \frac{1}{y_{2,1}} \right) T^{2 - 2t} + \left( \frac{1}{y_{1,3}} + \frac{y_{1,3}}{y_{1,2}} + \frac{y_{2,1}}{y_{3,1}} + y_{3,1} \right) T^{2+t}.
\end{equation}
We decompose~\eqref{equ_potentialfl134} along the guidance of the subgraph $\Gamma_{f_1}$ associated to $f_1$ in Theorem~\ref{theorem_AnChoKim} as in Figure~\ref{Conicfib}. Taking a bulk-parameter $\frak{b}$ as in~\eqref{equ_bulkparam} and setting $c \coloneqq \exp(\frak{c})$ and $\underline{c} \coloneqq \exp(\underline{\frak{c}})$, we have the following system of equations.
\begin{equation*}
\begin{cases}
\displaystyle W_\frak{b}^{\ydiagram{1,2}}({\bf{y}}) \coloneqq \frac{y_{1,2}}{y_{1,1}} + y_{1,2} + \frac{y_{1,1}}{y_{2,1}} + \frac{1}{y_{2,1}}\\ 
\displaystyle W_{\frak{b}{\hphantom{\ydiagram{1}}}}^{\ydiagram{1,1}}({\bf{y}}) \coloneqq c \cdot \frac{1}{y_{1,3}} + \frac{y_{1,3}}{y_{1,2}} + a \cdot y_{1,2}  \\  
\displaystyle W_\frak{b}^{\ydiagram{2}} ({\bf{y}}) \coloneqq \underline{c} \cdot {y_{3,1}} + \frac{y_{2,1}}{y_{3,1}} + \underline{a} \cdot \frac{1}{y_{2,1}}.
\end{cases}
\end{equation*}
Then the split leading term equation is the system of equations consisting of the logarithmic derivatives of the decomposed potential functions, that is, 
$$
\begin{cases}
\pa_{(1,1)} W^{\ydiagram{1,2}}_\frak{b} = 0, \,\, \pa_{(1,2)} W^{\ydiagram{1,2}}_\frak{b} = 0, \,\, \pa_{(2,1)} W^{\ydiagram{1,2}}_\frak{b} = 0, \\ 
\pa_{(1,2)} W^{\ydiagram{1,1}}_\frak{b} = 0,\,\, \pa_{(1,3)} W^{\ydiagram{1,1}}_\frak{b}= 0, \\
\pa_{(2,1)} W^{\ydiagram{2}}_\frak{b} = 0, \,\, \pa_{(3,1)} W^{\ydiagram{2}}_\frak{b} = 0, 
\end{cases}
$$
where $\pa_{(i,j)} \coloneqq y_{i,j} \frac{\partial}{\partial y_{i,j}}$. This system has a complex solution $$
y_{1,1} = y_{1,2} = y_{2,1} = -1, \, y_{1,3} = 1, \, y_{3,1} = -1, \, c = -1, \, \underline{c} = -1, \, a = 1, \mbox{ and } \underline{a} = -1.
$$
This complex solution extends to a critical point of the bulk-deformed potential later on.

\begin{figure}[ht]
	\scalebox{1}{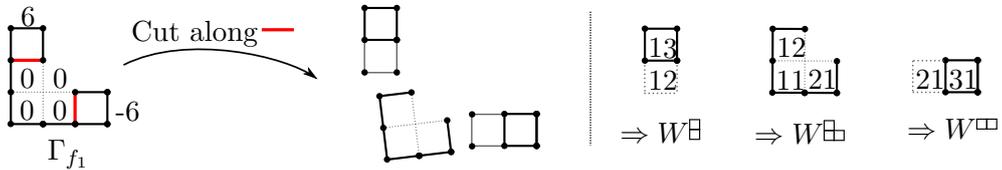}
	\caption{\label{Conicfib} Decomposition of the ladder diagram $\Gamma_3$ by $\Gamma_{f_1}$.}	
\end{figure}

For the general case where $n \geq 3$, we decompose $W_\mathfrak{b}({\bf{y}})$ into 
\begin{equation}\label{equ_decomposedpotent}
\begin{cases}
\displaystyle W_\frak{b}^{\ydiagram{1,2}}({\mathbf{y}}) = \frac{y_{1,2}}{y_{1,1}} + y_{1,2} + \frac{y_{1,1}}{y_{2,1}} + \frac{1}{y_{2,1}}, \\ 
\displaystyle W_{\frak{b}{\hphantom{\ydiagram{1}}}}^{\ydiagram{1,1}}({\mathbf{y}}) = c \cdot \frac{1}{y_{1,n}} + \frac{y_{1,n}}{y_{1,n-1}} + \cdots + \frac{y_{1,3}}{y_{1,2}} + {a} \cdot {y_{1,2}}, \\  
\displaystyle W_\frak{b}^{\ydiagram{2}} ({\bf{y}}) =  \underline{c} \cdot {y_{n,1}} + \frac{y_{n-1,1}}{y_{n,1}} + \cdots + \frac{y_{2,1}}{y_{3,1}} + \underline{a} \cdot \frac{1}{y_{2,1}}.
\end{cases}
\end{equation}
The split leading term equation comes from the logarithmic derivatives of the decomposed potential functions in~\eqref{equ_decomposedpotent}.

\begin{lemma}\label{lemma_solsplitleadingterm}
The split leading term equation has a complex solution in $(\C^*)^{2n-1}_{\mathbf{y}} \times (\C^*)^2_{c, \underline{c}}$ with a choice of $a = 1$ and $\underline{a} = -1$.
\end{lemma}

\begin{proof}
Taking $y_{1,1} = y_{1,2} = y_{2,1} = y_{3,1} = - 1$ and $y_{1,3} = 1$, every logarithmic derivative $W_\frak{b}^{\ydiagram{1,2}}({\mathbf{y}})$ vanishes. Moreover, the derivatives of $W_{\frak{b}{\hphantom{\ydiagram{1}}}}^{\ydiagram{1,1}}({\mathbf{y}})$ and $W_\frak{b}^{\ydiagram{2}} ({\bf{y}})$ with respect to $y_{1,2}$ and $y_{2,1}$ also vanish. The split leading term equation gives rise to
\begin{equation}\label{equ_solsplitleading}
y_{1,j+1} = \frac{y_{1,j}^2}{y_{1,j-1}}, y_{j+1,1} = \frac{y_{j,1}^2}{y_{j-1,1}}, c = \frac{y_{1,4}^2}{y_{1,3}}, \underline{c} = \frac{y_{3,1}}{y_{4,1}^2} 
\end{equation}
The remaining variables can be inductively determined. It is straightforward to see that $(y_{1,j} = (-1)^{j-1}, y_{j, 1} = \underline{c} = -1, c =  (-1)^n)$ is a solution of~\eqref{equ_solsplitleading}.
\end{proof}

We are ready to prove Proposition~\ref{proposition_criticalpoints}.

\begin{proof}[Proof of Proposition~\ref{proposition_criticalpoints}]
We extend a solution of the split leading term equation in Lemma~\ref{lemma_solsplitleadingterm} to a critical point of the bulk-deformed potential function $W_\frak{b}(\mathbf{y})$ in~\eqref{equ_potentialfl1nn1} over $\mathrm{U}(\Lambda)$. For notational simplicity, we declare 
$$
\partial^\frak{b}_{(i,j)} (\mathbf{y}) \coloneqq  y_{i,j} \frac{\partial}{\partial y_{i,j}} W_\frak{b}(\mathbf{y}).
$$

We begin by extending the complex solution of the split leading term equation as 
$$
y_{1,1} = y_{1,2} = y_{2,1} = - 1 - T^{nt} \in \mathrm{U}(\Lambda),
$$
reflecting our choice of $a$. Then we have $\partial^\frak{b}_{(1,1)}(\mathbf{y}) = 0$. Also,
\begin{align}
&T^{(n-1)(t-1)} \cdot \partial_{(1,2)}^\frak{b}(\mathbf{y}) = 1 - 1 - T^{nt}  + y_{1,3} T^{nt}
= (y_{1,3} - 1) T^{nt} \\
&T^{(n-1)(t-1)} \cdot \partial_{(2,1)}^\frak{b}(\mathbf{y}) = -1 + 1 - T^{nt} - \frac{1}{y_{3,1}} T^{nt}
= - \left(1 + \frac{1}{y_{3,1}} \right) T^{nt}
\end{align}
mod $\Lambda_{> nt}$. We then solve $y_{1,3}$ (resp. $y_{3,1}$) such that $\partial_{(1,2)}^\frak{b}(\mathbf{y}) = 0$ (resp. $\partial_{(2,1)}^\frak{b}(\mathbf{y}) = 0$) holds. Note that
$y_{1,3} = 1$ and $y_{3,1} = -1$ mod $\Lambda_{>0}$, which agrees with a solution of the split leading term equation with the choice of $a= 1, \underline{a} = -1$. In particular, both $y_{1,3}$ and $y_{3,1}$ are in $\mathrm{U}(\Lambda)$.

Suppose that we have successively solved the equations
$$
\partial^\frak{b}_{(1,s)}(\mathbf{y}) = 0 \mbox{ and } \partial^\frak{b}_{(s,1)}(\mathbf{y}) = 0 \quad \mbox{ for $s = 1, 2, \cdots, j-1$}
$$
by determining $y_{1,1}, y_{1,2}, y_{2,1}, \cdots, y_{1,j}, y_{j,1} \in \mathrm{U}(\Lambda)$. The equations 
\begin{align}
&\partial^\frak{b}_{(1,j)}(\mathbf{y}) = 0 \,\, \Rightarrow \,\, y_{1,j+1} = \frac{y_{1,j}^2}{y_{1,j-1}}  \\
&\partial^\frak{b}_{(j,1)}(\mathbf{y}) = 0 \,\, \Rightarrow \,\, y_{j+1,1} = \frac{y_{j,1}^2}{y_{j-1,1}} 
\end{align}
and the pre-determined solutions determine $y_{1,j+1}$ and $y_{j+1,1}$. As 
\begin{equation}
y_{1,2s+1} = 1, \,\, y_{1,2s} = y_{2s, 1} = y_{2s+1,1} = -1 \mbox{ mod $\Lambda_{>0}$ for $j \geq 1$},
\end{equation} 
we have $y_{1,2s+3} = 1, y_{1, 2s+2} = y_{2s+2, 1}= y_{2s+3,1} = -1$ mod $\Lambda_{>0}$ as in Lemma~\ref{lemma_solsplitleadingterm}. In particular, $y_{1,j+1}, y_{j+1,1} \in \mathrm{U}(\Lambda)$. 
 
So far, we have exhausted all variables $y_{i,j} \in \mathrm{U}(\Lambda)
$ making $\partial^\frak{b}_{(i,j)}(\mathbf{y}) = 0$. Lastly, by taking the bulk-parameters
$$
c = \frac{y_{1,n}^2}{y_{1,n-1}} \in \mathrm{U}(\Lambda) \, \mbox{ and } \, \underline{c} = \frac{y_{n-1,1}}{y_{n,1}^2} \in \mathrm{U}(\Lambda), 
$$ 
one can make the last two equations $\partial^\frak{b}_{(1,n)}(\mathbf{y}) = 0$ and $\partial^\frak{b}_{(n,1)}(\mathbf{y}) = 0$ hold.
\end{proof}


\providecommand{\bysame}{\leavevmode\hbox to3em{\hrulefill}\thinspace}
\providecommand{\MR}{\relax\ifhmode\unskip\space\fi MR }
\providecommand{\MRhref}[2]{%
  \href{http://www.ams.org/mathscinet-getitem?mr=#1}{#2}
}
\providecommand{\href}[2]{#2}

\end{document}